\documentclass[12pt]{article}
\usepackage[utf8]{inputenc}
\usepackage{titlesec}
\usepackage{geometry}
\usepackage{authblk}
\usepackage{amsmath}
\usepackage{amssymb}
\usepackage{latexsym}
\usepackage[table]{xcolor}
\usepackage{bm}
\usepackage{pstricks}
\usepackage{pst-plot}
\usepackage{psfrag}
\usepackage{graphicx}
\usepackage{multirow}
\usepackage{multicol}
\usepackage{caption}
\usepackage{floatrow, subcaption}
\usepackage{tikz}
\usepackage{setspace}
\usepackage{lineno}
\usepackage[english]{babel}
\usepackage{amsthm}
\usepackage{tabularx, booktabs}
\usepackage{changepage}
\usepackage{biblatex}
\definecolor{tableShade}{gray}{0.9}
\newcolumntype{Y}{>{\centering\arraybackslash}X}
\newcolumntype{Z}{>{\raggedleft\arraybackslash}X}
\newcolumntype{S}{>{\hsize=.2\hsize}X}

\newtheorem{theorem}{Theorem}[section]

\newtheorem{lemma}[theorem]{Lemma}
\newtheorem{proposition}[theorem]{Proposition}

\numberwithin{equation}{section}

\titleformat*{\section}{\normalsize\bfseries}
\titleformat*{\subsection}{\normalsize\bfseries}
\titleformat*{\subsubsection}{\normalsize\bfseries}
\titleformat*{\paragraph}{\normalsize\bfseries}
\titleformat*{\subparagraph}{\normalsize\bfseries}

\titlespacing\section{0pt}{12pt plus 4pt minus 2pt}{0pt plus 2pt minus 2pt}
\titlespacing\subsection{0pt}{12pt plus 4pt minus 2pt}{0pt plus 2pt minus 2pt}
\titlespacing\subsubsection{0pt}{12pt plus 4pt minus 2pt}{0pt plus 2pt minus 2pt}

\title{\vspace{-0.5in}N-BODY APPROACH TO THE TRAVELING SALESMAN PROBLEM (TSP)}
\author{Johnny Seay$^1$, Edwin Gonzalez, Stephen Lowe, \\Dr. Jesse Crawford, Dr. Byrant Wyatt}
\date{}

\begin{document}

% \maketitle
\begingroup
    \centering
    \centerline{N-BODY APPROACH TO THE TRAVELING SALESMAN PROBLEM (TSP)}
    \bigskip
    Johnny Seay$^1$, Edwin Gonzalez$^2$, Stephen Lowe$^3$,\\
    \centerline{Dr. Jesse Crawford$^4$, Dr. Byrant Wyatt$^5$}
    Department of Mathematics, Tarleton State University\\
    Box T-0470, Stephenville,TX 76402\\
    $^1$johnny.seay@go.tarleton.edu\\
    $^2$eegonzalez@atkore.com\\
    $^3$stephen.lowe@go.tarleton.edu\\
    $^4$jcrawford@tarleton.edu\\
    $^5$wyatt@tarleton.edu\\
\endgroup
\bigskip
\bigskip

\centerline{\textbf{Abstract}}
\begin{adjustwidth}{2.5em}{2.5em}
In the Traveling Salesman Problem (TSP), a list of cities and the distances between them are given.  The goal is to find the shortest possible route that visits each city exactly once and returns to the original city. The TSP has a wide range of applications in many different industries including, but not limited to, optimizing mail and shipping routes, guiding industrial machines, mapping genomes, and improving autonomous vehicles. For centuries, traveling salesmen, politicians, and circuit preachers have tackled their own versions of the problem. Within the last century, the TSP has become one of the most important problems in the fields of mathematics and computer science. The time to find an exact solution is often impractically long, which has led to the development of numerous approximation techniques, ranging from linear programming methods to nature-inspired models. Here, we present a novel N-body approach to the TSP.
\end{adjustwidth}

\section{Introduction}
\subsection{Background}
Before the invention of railroads or automobiles, traveling was a demanding and time consuming part of life. What would take us hours or days would take them weeks, months, or even years. Because of this, some professions (such as traveling salesmen, politicians, and circuit preachers) greatly benefited from carefully planned routes \cite{cook}. However, when narrowing the scope to just mathematical history, touring problems have been studied since the mid-1700s, when Leonhard Euler presented his famous Seven Bridges of Königsberg problem to the St. Petersburg Academy \cite{euler}. Despite this, it took roughly 200 years before the Traveling Salesman Problem received its first mathematical consideration when Merrill Flood was looking to solve a school bus routing problem \cite{lawler}. Since then, the TSP has become a popular problem in mathematics and computer science. It is one of the most intensively studied problems in optimization and is often used as a benchmark for optimization methods.

\subsection{Approximation Algorithms}
Running times for exact algorithms have a lower bound of $O(n^22^n)$ and an upper bound of $O(n!)$ \cite{woeginger}. For many real-world applications of the Traveling Salesman Problem, these running times are impractical. As a consequence, it may be more beneficial to work with a good approximation than to spend the time finding the exact solution. Because of this, approximation methods and algorithms are often chosen over exact algorithms. Some of these include the nearest-neighbor algorithm, the convex-hull-and-line algorithm, the Christofides algorithm, and Ant-Colony optimization. The nearest-neighbor simply has the traveling salesman pick the closest unvisited city to move to next. The convex-hull-and-line algorithm considers the fact that ``in the euclidean plane the minimal (or optimal) tour does not intersect itself'' in order to find solutions \cite{deineko}. The Christofides algorithm combines a minimum spanning tree and a minimum-weight perfect matching to produce approximate solutions \cite{goodrich}. A nature-inspired method, the Ant Colony optimization models the observation that ants prefer to follow trails containing pheromones deposited by other ants \cite{dorigo}. Of course, there are many other approaches and algorithms inspired from all facets of life, nature, and mathematics. 

\subsection{N-Body Simulation}
Before we delve into how we use N-body simulations in our approach to the Traveling Salesman Problem, it may be useful to describe what an N-body simulation is. In essence, an N-body simulation is a dynamical system of particles. These particles interact with each other through forces (e.g. gravitational or spring forces). The particles may also be influenced by additional external forces. Typically, N-body simulations are used to simulate physical processes with the scale of such phenomena ranging from intramolecular dynamics to galaxy formations \cite{greenspan}. However, we will be using N-body simulations to solve a combinatorial optimization problem. In simpler terms, and to highlight the novelty of our approach, we are using physics to find solutions to an abstract mathematical problem.

\section{Methodology}
\subsection{General/Simple}

With our N-body approach, each city is treated as a particle. The particles interact with each other through an attractive-repulsive force; for this we use a Lennard-Jones type force. To briefly describe this interaction, consider a pair of particles and the initial distance between them. This initial distance will be referred to as the natural distance. If the two particles are further apart than their natural distance, they will be attracted to each other. If the particles are closer together than their natural distance, they will be repulsed by each other. To penalize moves between distant cities, the magnitude of the repulsive force is much greater than that of the attractive force. For a more detailed explanation and analysis of the Lennard-Jones type force functions, see Appendix B.

\bigskip

We define the origin of the system to be the geometric center (or center of mass) of the particles. To keep the system contained, the particles will be surrounded by a minimum bounding circle, centered at the origin (see Figure \ref{fig:generalA}). This bounding circle acts as a wall, pushing particles inward. We refer to this bounding circle as the outer wall. Additionally, a circle with an initial radius of zero will be placed at the origin (see Figure \ref{fig:generalB}). This inner circle works as a wall pushing particles outward and will be referred to as the inner wall.

\bigskip

Over time, this inner wall will grow, increasing its radius until it reaches the outer wall. As the inner wall grows, it will push against the particles it comes into contact with. In turn, those particles will push or pull on the other particles. This process is visualized in Figures \ref{fig:generalC}-\ref{fig:generalF}. By the time the inner wall has reached the outer wall, all of the particles will have been forced into a ring, trapped between the inner and outer walls (see Figure \ref{fig:generalG}).  We observe the order in which the particles fall on this ring to obtain a path in the initial city configuration (see Figure \ref{fig:generalH}). The outer and inner circles being treated as walls acting on the particles allow us to squeeze this two-dimensional system into a one-dimensional path. The motivation behind this approach is that the system will try to minimize its energy as the particles are being forced into a ring. The idea is that a connection between the minimal energy state of the ring and an optimal path exists.
\begin{figure}[htbp!]
\centering
\begin{subfigure}{.22\linewidth}
    \centering
    \includegraphics[scale=0.5]{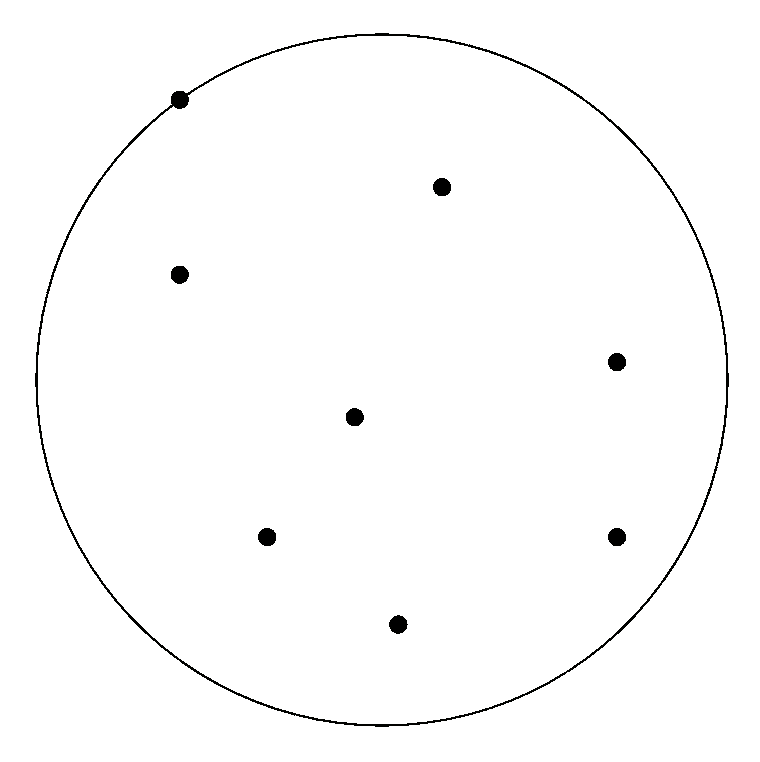}
    \caption{}\label{fig:generalA}
\end{subfigure}
    \hfill
\begin{subfigure}{.22\linewidth}
    \centering
    \includegraphics[scale=0.5]{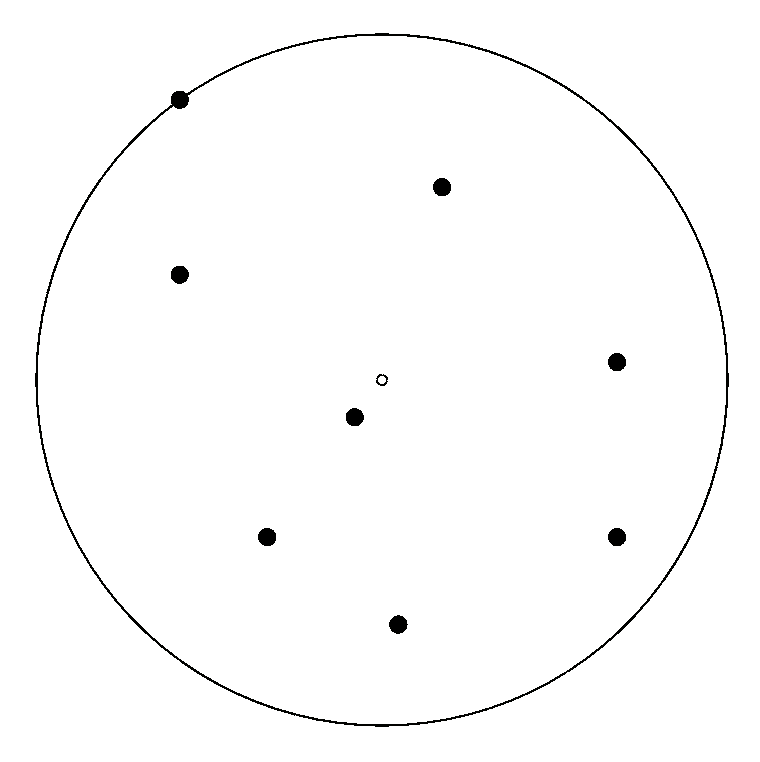}
    \caption{}\label{fig:generalB}
\end{subfigure}
   \hfill
\begin{subfigure}{.22\linewidth}
    \centering
    \includegraphics[scale=0.5]{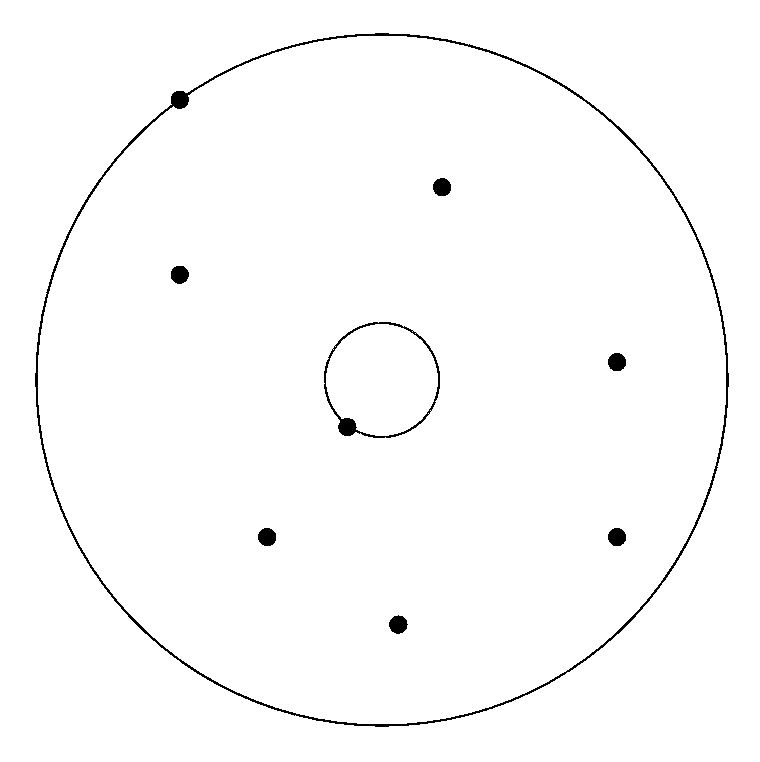}
    \caption{}\label{fig:generalC}
\end{subfigure}
    \hfill
\begin{subfigure}{.22\linewidth}
    \centering
    \includegraphics[scale=0.5]{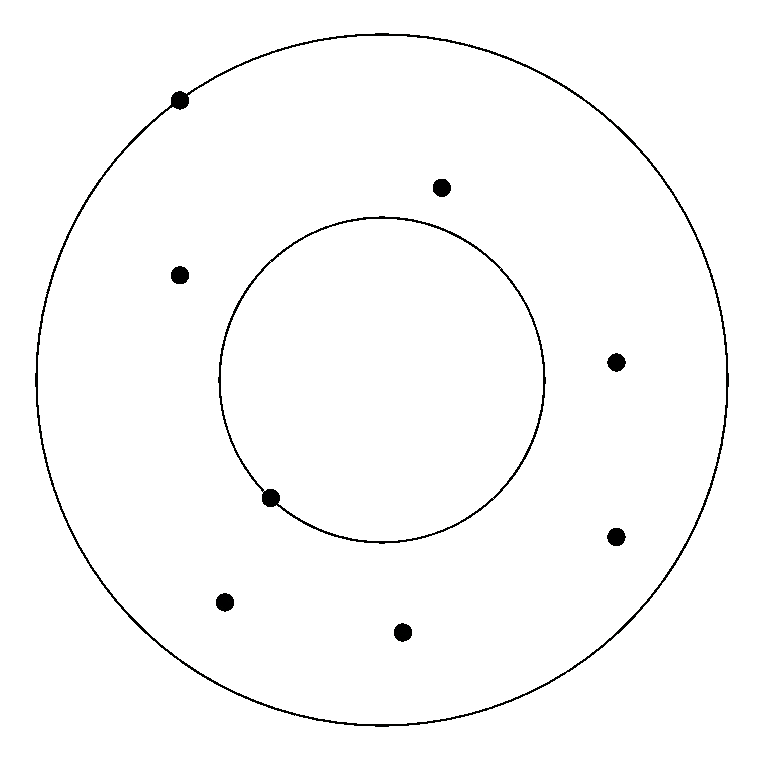}
    \caption{}\label{fig:generalD}
\end{subfigure}
\bigskip
\begin{subfigure}{.22\linewidth}
    \centering
    \includegraphics[scale=0.5]{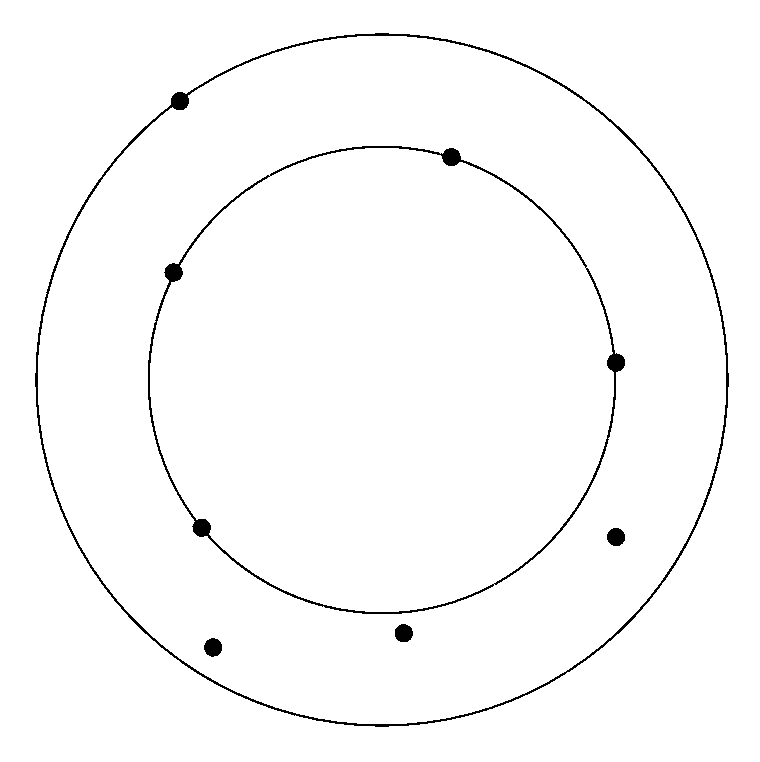}
    \caption{}\label{fig:generalE}
\end{subfigure}
    \hfill
\begin{subfigure}{.22\linewidth}
    \centering
    \includegraphics[scale=0.5]{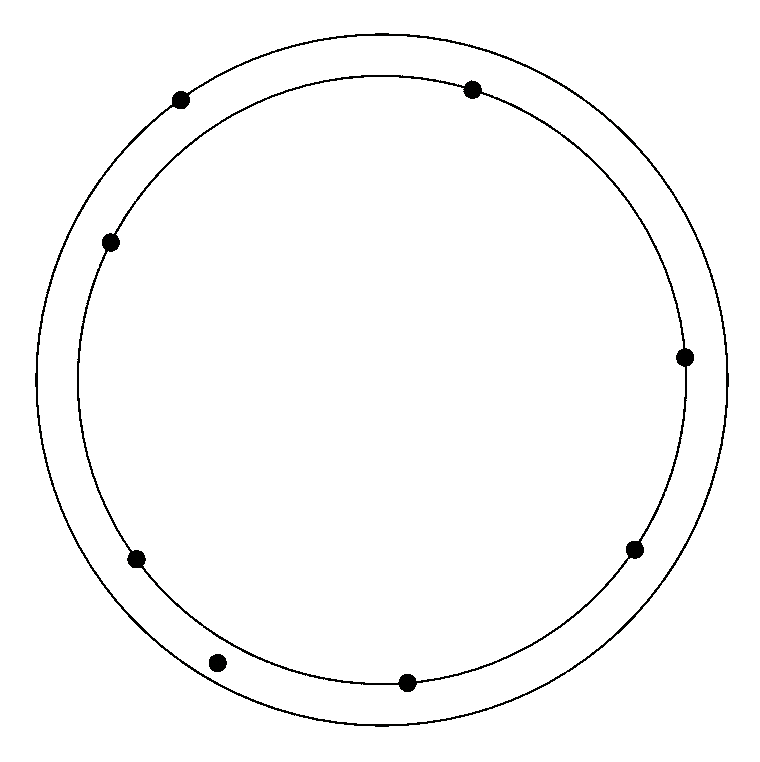}
    \caption{}\label{fig:generalF}
\end{subfigure}
   \hfill
\begin{subfigure}{.22\linewidth}
    \centering
    \includegraphics[scale=0.5]{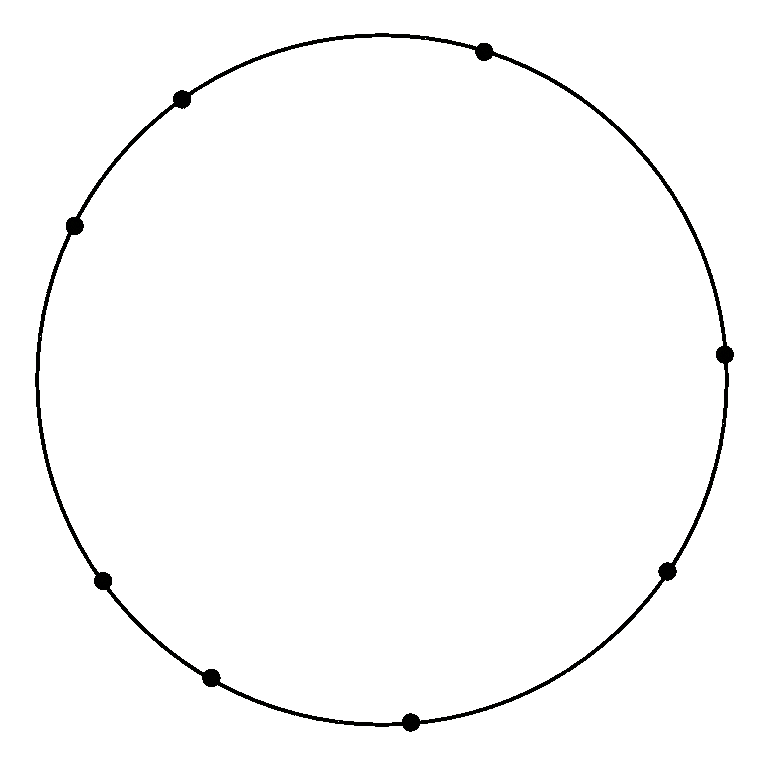}
    \caption{}\label{fig:generalG}
\end{subfigure}
    \hfill
\begin{subfigure}{.22\linewidth}
    \centering
    \includegraphics[scale=0.5]{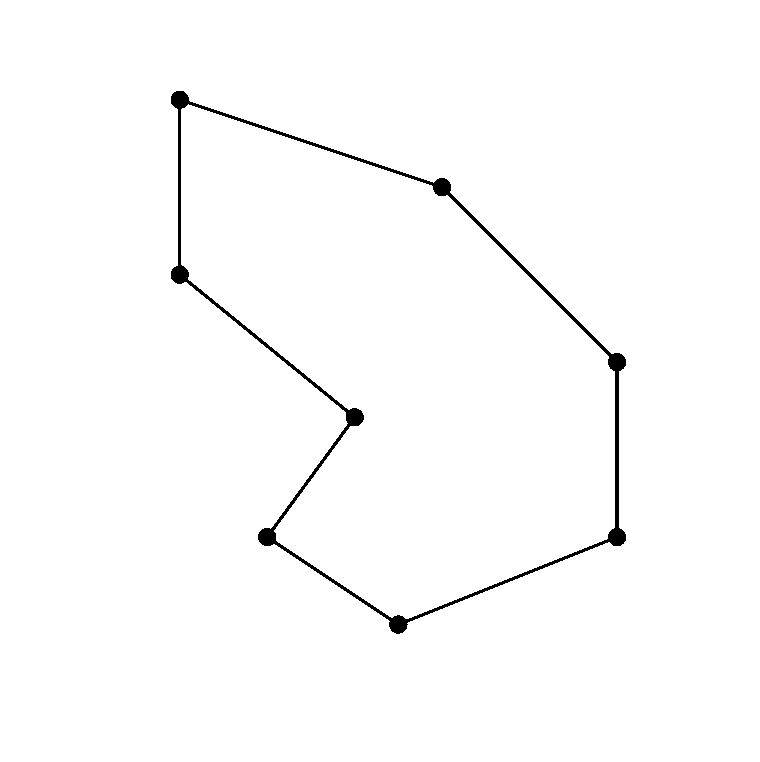}
    \caption{}\label{fig:generalH}
\end{subfigure} 

\RawCaption{\caption{Phases of the general N-body simulation.}
\label{fig:general}}
\end{figure}

\subsection{Preliminary Results}

In order to determine if this approach was feasible, we first ran our N-body simulation on relatively small sets of independent randomly generated instances and compared our results with the exact cost and the nearest-neighbor algorithm. These results are shown in Table \ref{tab:results-random}. The data for these results can be found online by following the link provided in the Supplemental Material section. To find the exact costs for these randomly generated instances we used the brute-force approach; hence the cutoff at 12 cities. The percent error presented in the results is a measurement of how close the approximated cost is to the exact cost, as a percentage of the exact cost (shown below).

$$\text{Percent Error} = \frac{\text{Approximated Cost} - \text{Exact Cost}}{\text{Exact Cost}} \times 100$$

Looking at the results presented in Table \ref{tab:results-random}, we saw potential in this approach. Moving on to slightly bigger instances, the N-body results for a 4x4 grid instance are presented in Figure \ref{fig:simple-grid-results}. Finding an optimal solution for full grid instances is relatively simple and can be done by hand, even for a large number of cities. For these instances, an optimal solution can be found by minimizing the number of diagonal edges within the path. Next, we tried our approach on the \textit{att48} instance. This instance consists of the capitals of the 48 contiguous states. The optimal cost for this instance is known. These results are presented in Figure \ref{fig:simple-att48-results}. The average runtime of the simulation for these specific instances are included in order to provide the reader with some insight into the time it takes to find a solution with this approach. More results for these two instances will be provided throughout this text. Results for additional instances are provided in Appendix A.

\begin{table}[H]
\centering
\caption{Results from 100 runs of independent randomly generated datasets.}
\label{tab:results-random}
\renewcommand{\tabularxcolumn }[1]{ >{\arraybackslash}b{#1}}
\rowcolors{1}{tableShade}{white}
\begin{tabularx}{0.95\textwidth}{@{}YYY@{}}
\rowcolor{white}
\textbf{Number of Cities} & \textbf{Average N-body Percent Error} & \textbf{Average Nearest Neighbor Percent Error}\\ \midrule
8 & 1.2311\% & 8.0474\%\\
9 & 2.5020\% & 9.7363\%\\
10 & 2.1861\% & 9.9390\%\\
11 & 2.4125\% & 11.2259\%\\
12 & 4.0270\% & 13.4567\%\\
\bottomrule
\end{tabularx}
\end{table}

\begin{figure}[H]
    \centering
    \begin{minipage}{0.25\textwidth}
        \centering
        \includegraphics[scale=0.75]{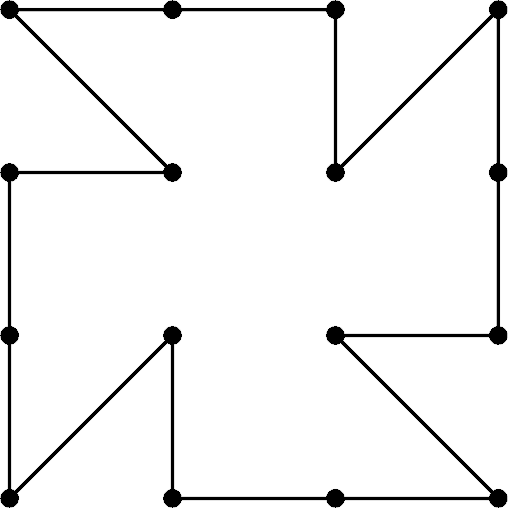}
    \end{minipage}
    \begin{minipage}{0.5\textwidth}
        \begin{footnotesize}
        \centering
        \renewcommand{\tabularxcolumn}[1]{ >{\arraybackslash}b{#1}}
        \begin{tabularx}{1.1\textwidth}{Z p{2cm}}
        Optimal Path Cost: & 16.000 \\
        Simple N-body Path Cost: & 17.657 \\
        \textbf{Percent Error:} & \textbf{10.355\%} \\ \cmidrule(lr){1-2}
        Average Runtime: & 6.67 s
        \end{tabularx}
        \end{footnotesize}
    \end{minipage}
    \caption{Simple N-body result for 4x4 grid instance.}
    \label{fig:simple-grid-results}
\end{figure}

\begin{figure}[H]
    \centering
    \begin{minipage}{0.25\textwidth}
        \centering
        \includegraphics[scale=0.75]{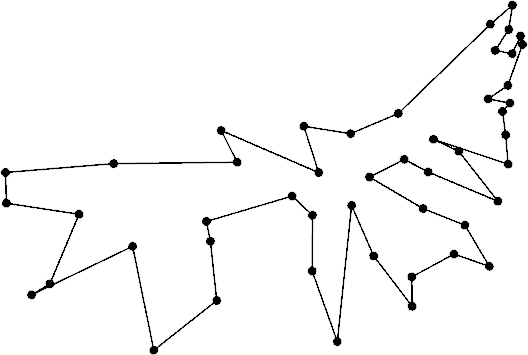}
    \end{minipage}
    \begin{minipage}{0.5\textwidth}
        \begin{footnotesize}
        \centering
        \renewcommand{\tabularxcolumn}[1]{ >{\arraybackslash}b{#1}}
        \begin{tabularx}{1.1\textwidth}{Z p{2cm}}
        Optimal Path Cost: & 33523.708 \\
        Simple N-body Path Cost: & 38862.859 \\
        \textbf{Percent Error:} & \textbf{15.927\%} \\ \cmidrule(lr){1-2}
        Average Runtime: & 6.69 s
        \end{tabularx}
        \end{footnotesize}
    \end{minipage}
    \caption{Simple N-body result for \textit{att48} instance.}
    \label{fig:simple-att48-results}
\end{figure}

\subsection{Variations/Features}
As the number of cities increases, the system may become dense. Consequently, the potential energy within the system increases. As a result, the system has a propensity to become increasingly more erratic as the area between the two walls approaches zero. This erratic behavior results in undesirable solutions. Another complication arises when dealing with non-uniform instances of the TSP. With non-uniform instances, denser groups of particles tend to move as a single unit. As these clumps of particles are squished between the walls, they exhibit similar erratic behavior to that mentioned previously but on a smaller scale. We have found that breaking these clumps up can yields better results. Addressing these erratic phenomena is crucial in optimizing our N-body approach.

\subsubsection{Pressure/Global Density}
As mentioned earlier, erratic behaviour appears when working with dense instances for two reasons. The particles do not have sufficient time to readjust themselves into a desirable configuration and the total force in the system can go beyond the limits of the numerical scheme as the separation of the walls approaches zero, resulting in the particles breaking through the walls. This can be seen in Figure \ref{fig:pressure1}. To address this, we initially added sufficient perimeter to the outer wall to handle the total force generated. However, this resulted in the entire system clustering to one side of the inner wall, as seen in Figure \ref{fig:pressure2}. This also produced undesirable results. Our next step was to allow the outer wall to adjust accordingly to the state of the system by measuring the forces exerted on the outer wall. We refer to the summation of these forces divided by the perimeter of the outer wall as pressure. If the pressure falls outside of an acceptable range, the outer wall will grow or shrink accordingly. This achieves the desirable effects of giving the system more time and room to move into lower energy states while maintaining the integrity of our numerical scheme. This can be seen in Figure \ref{fig:pressure3}.

\begin{figure}[H]
\centering
\begin{subfigure}{.22\linewidth}
    \centering
    \includegraphics[scale=0.5]{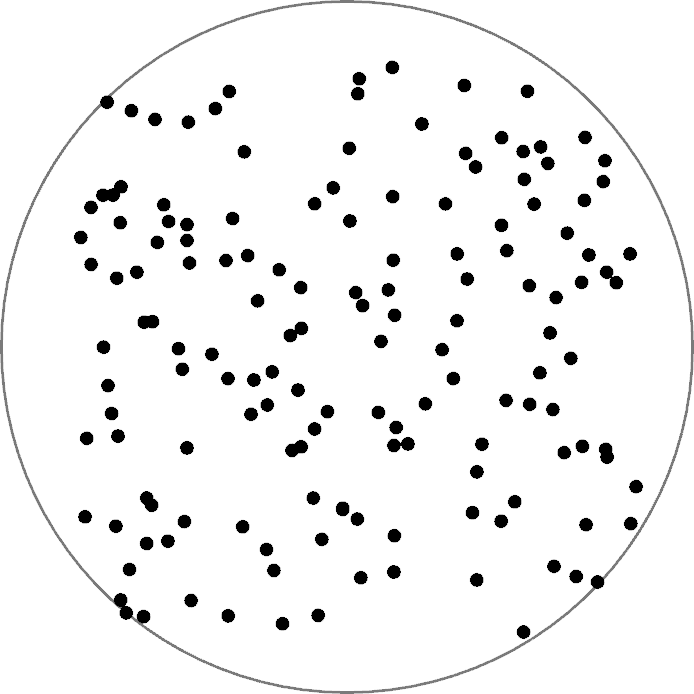}
    \caption{}\label{fig:pressureA}
\end{subfigure}
\hfill
\begin{subfigure}{.22\linewidth}
    \centering
    \includegraphics[scale=0.65]{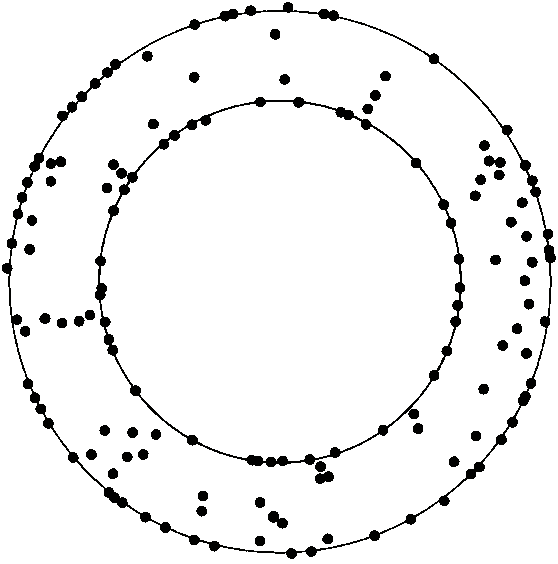}
    \caption{}\label{fig:pressureB}
\end{subfigure}
\hfill
\begin{subfigure}{.22\linewidth}
    \centering
    \includegraphics[scale=0.65]{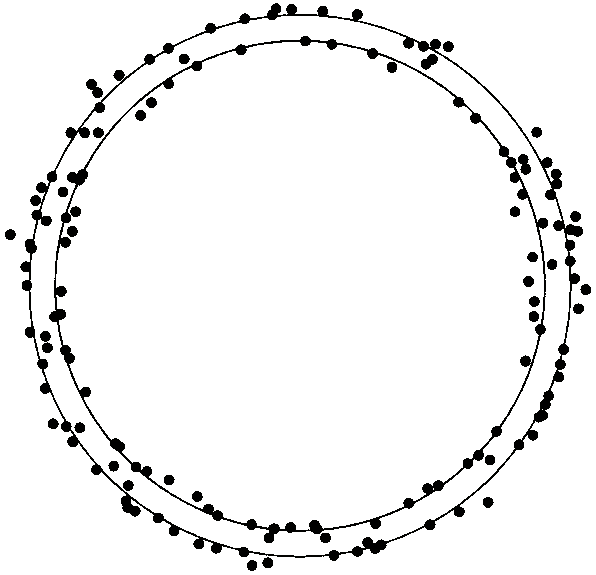}
    \caption{}\label{fig:pressureC}
\end{subfigure}
\hfill
\begin{subfigure}{.22\linewidth}
    \centering
    \includegraphics[scale=0.65]{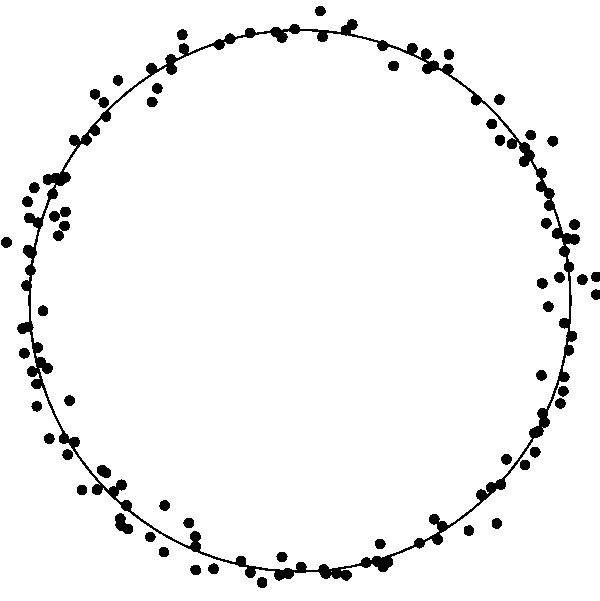}
    \caption{}\label{fig:pressureD}
\end{subfigure}
\RawCaption{\caption{Dense instance breaking through the walls.}
\label{fig:pressure1}}
\end{figure}

\begin{figure}[H]
\centering
\begin{subfigure}{.22\linewidth}
    \centering
    \includegraphics[scale=0.5]{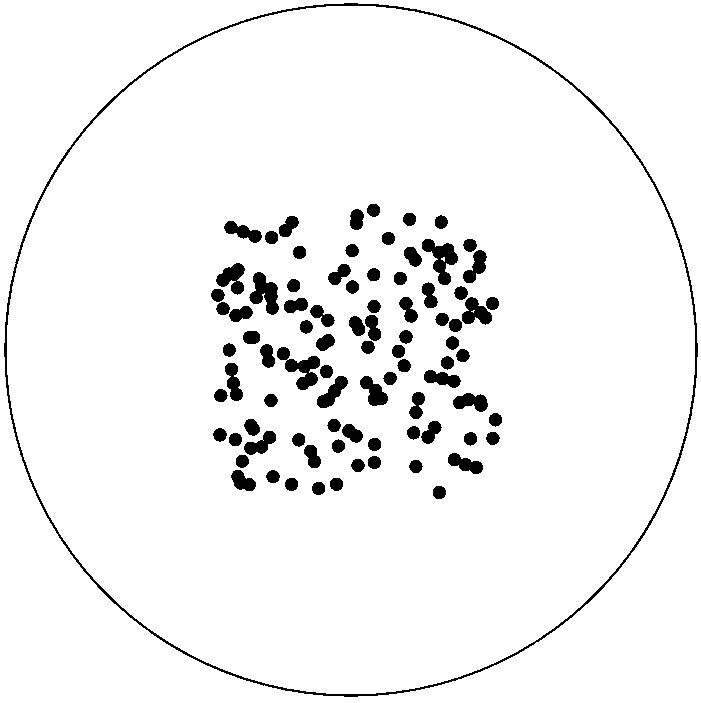}
    \caption{}\label{fig:pressureE}
\end{subfigure}
\hfill
\begin{subfigure}{.22\linewidth}
    \centering
    \includegraphics[scale=0.5]{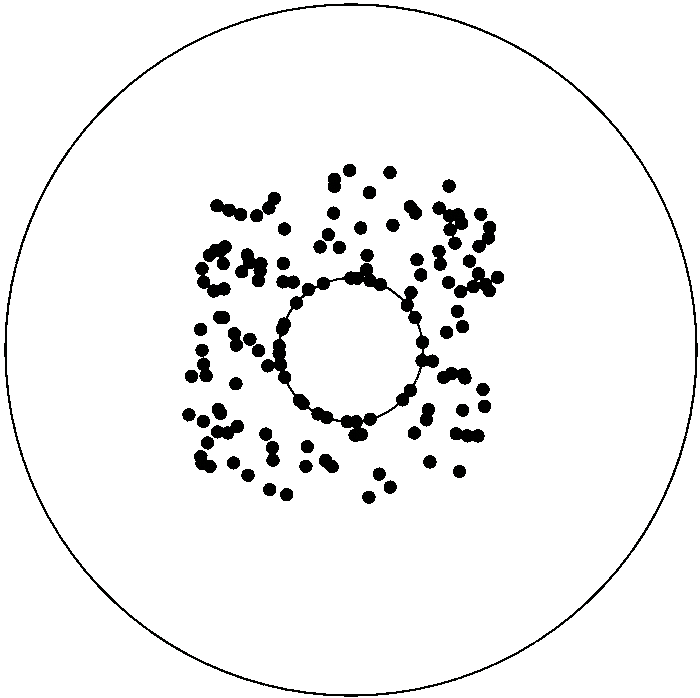}
    \caption{}\label{fig:pressureF}
\end{subfigure}
\hfill
\begin{subfigure}{.22\linewidth}
    \centering
    \includegraphics[scale=0.5]{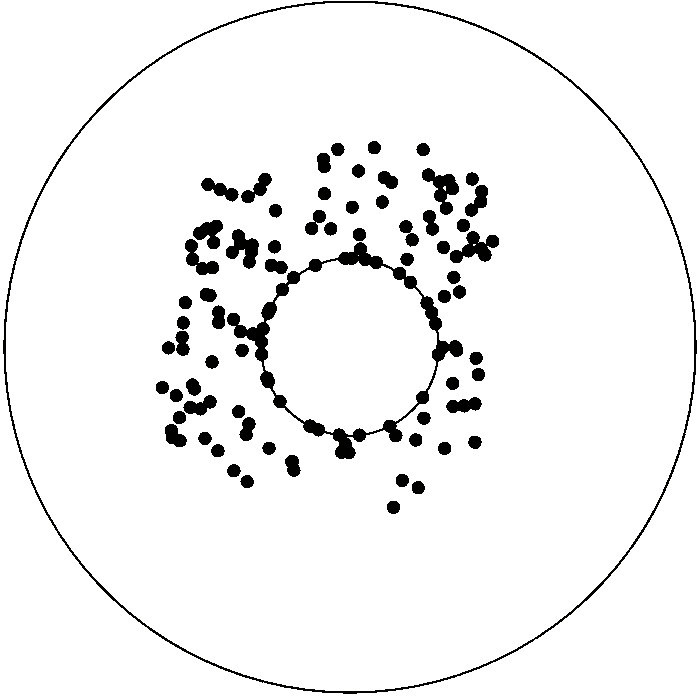}
    \caption{}\label{fig:pressureG}
\end{subfigure}
\hfill
\begin{subfigure}{.22\linewidth}
    \centering
    \includegraphics[scale=0.5]{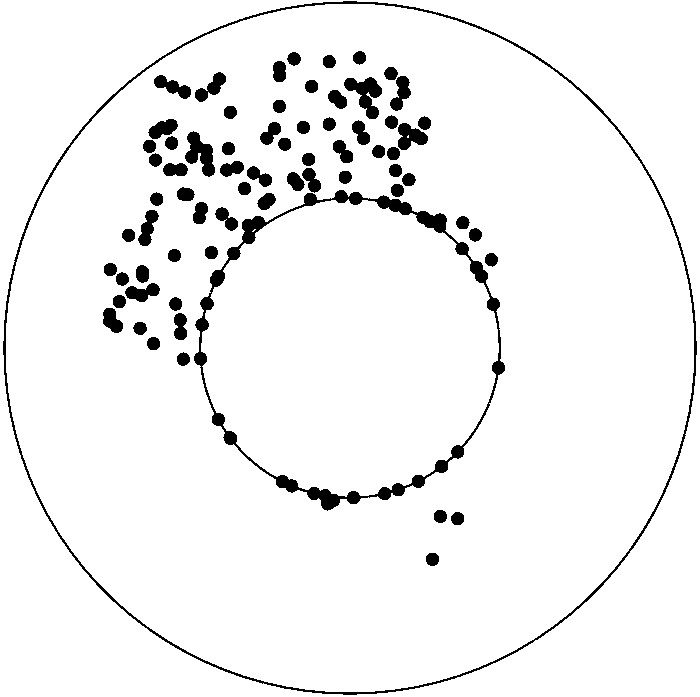}
    \caption{}\label{fig:pressureH}
\end{subfigure}
\RawCaption{\caption{Starting with a larger initial outer wall radius.}
\label{fig:pressure2}}
\end{figure}

\begin{figure}[H]
\centering
\begin{subfigure}{.22\linewidth}
    \centering
    \includegraphics[scale=0.5]{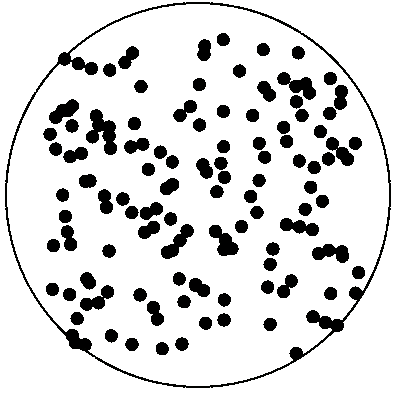}
    \caption{}\label{fig:pressureI}
\end{subfigure}
\hfill
\begin{subfigure}{.22\linewidth}
    \centering
    \includegraphics[scale=0.5]{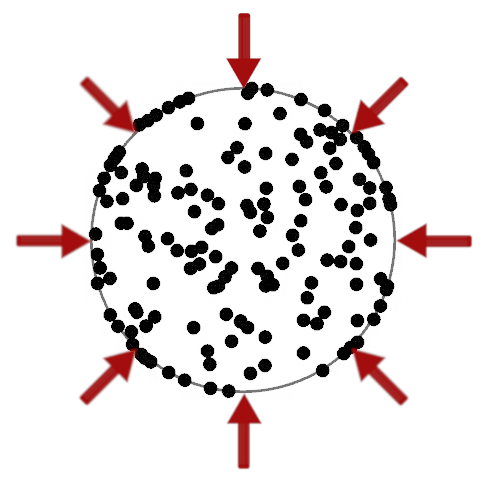}
    \caption{}\label{fig:pressureJ}
\end{subfigure}
\hfill
\begin{subfigure}{.22\linewidth}
    \centering
    \includegraphics[scale=0.5]{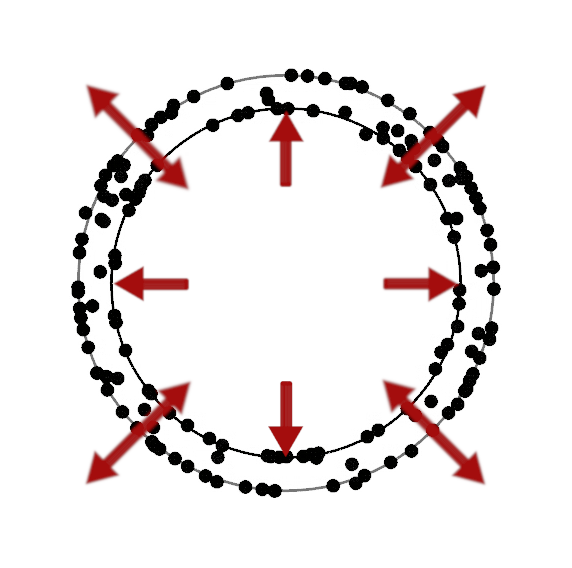}
    \caption{}\label{fig:pressureK}
\end{subfigure}
\hfill
\begin{subfigure}{.22\linewidth}
    \centering
    \includegraphics[scale=0.5]{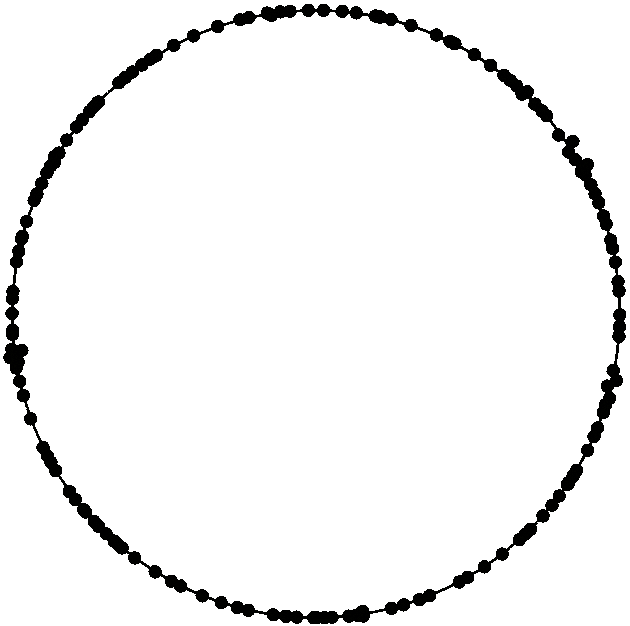}
    \caption{}\label{fig:pressureL}
\end{subfigure}

\RawCaption{\caption{Letting the outer wall adjust as needed.}
\label{fig:pressure3}}
\end{figure}

Some results of our N-body simulation with pressure implemented are shown in Figure \ref{fig:pressure-grid-results} and Figure \ref{fig:pressure-att48-results}.  For the 4x4 grid instance, this pressure variation resulted in a different path but yielded the same cost as the simple N-body. However, for the \textit{att48} instance, this pressure variation resulted in a significant improvement over the results produced by the simple N-body approach. Additional results showcasing the improvements provided by adding pressure can be found in Appendix A, Figures 16-18.

\begin{figure}[H]
    \centering
    \begin{minipage}{0.25\textwidth}
        \centering
        \includegraphics[scale=0.75]{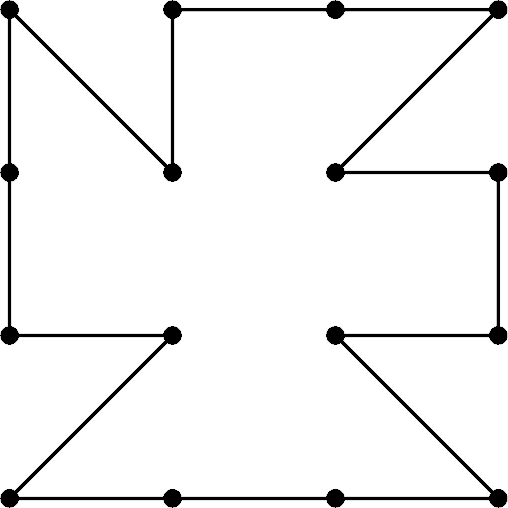}
    \end{minipage}
    \begin{minipage}{0.5\textwidth}
        \begin{footnotesize}
        \centering
        \renewcommand{\tabularxcolumn}[1]{ >{\arraybackslash}b{#1}}
        \begin{tabularx}{1.1\textwidth}{Z p{2cm}}
        Optimal Path Cost: & 16.000 \\
        Pressure N-body Path Cost: & 17.657 \\
        \textbf{Percent Error:} & \textbf{10.355\%} \\ \cmidrule(lr){1-2}
        Average Runtime: & 10.03 s
        \end{tabularx}
        \end{footnotesize}
    \end{minipage}
    \caption{Pressure N-body result for 4x4 grid instance.}
    \label{fig:pressure-grid-results}
\end{figure}

\begin{figure}[H]
    \centering
    \begin{minipage}{0.25\textwidth}
        \centering
        \includegraphics[scale=0.75]{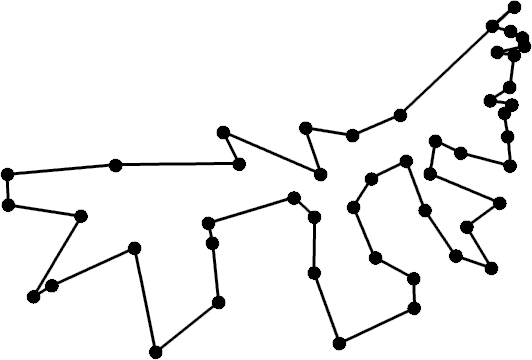}
    \end{minipage}
    \begin{minipage}{0.5\textwidth}
        \begin{footnotesize}
        \centering
        \renewcommand{\tabularxcolumn}[1]{ >{\arraybackslash}b{#1}}
        \begin{tabularx}{1.1\textwidth}{Z p{2cm}}
        Optimal Path Cost: & 33523.708 \\
        Pressure N-body Path Cost: & 36967.234 \\
        \textbf{Percent Error:} & \textbf{10.272\%} \\ \cmidrule(lr){1-2}
        Average Runtime: & 13.42 s
        \end{tabularx}
        \end{footnotesize}
    \end{minipage}
    \caption{Pressure N-body result for \textit{att48} instance.}
    \label{fig:pressure-att48-results}
\end{figure}

\subsubsection{Clumping/Local Density (``Bubble" Method)}
As stated earlier, non-uniform instances can result in localized clustering which can produce undesirable results (see Figure \ref{fig:density1}). In order to address this issue, additional circles are inserted in the denser areas of the system in an attempt to break them apart. These additional circles also act as walls pushing on the particles; we refer to these additional circles as bubbles to differentiate them from the outer and inner walls. To decide where to insert these bubbles, a rough density map is created by partitioning the space into cells and counting the number of particles in each cell. For each non-empty cell, the center of mass of the particles within that cell is calculated. For cells with density above a given value, their center of masses are used as the insertion points for the bubbles. This process is visualized in Figure \ref{fig:density2}.

\begin{figure}[H]
\centering
\begin{subfigure}{.22\linewidth}
    \centering
    \includegraphics[scale=0.5]{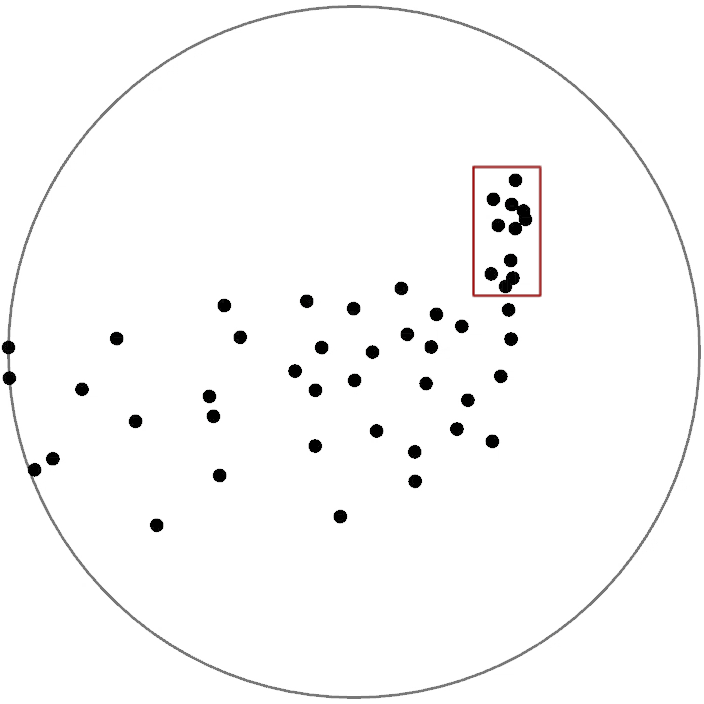}
    \caption{}\label{fig:densityA}
\end{subfigure}
\hfill
\begin{subfigure}{.22\linewidth}
    \centering
    \includegraphics[scale=0.5]{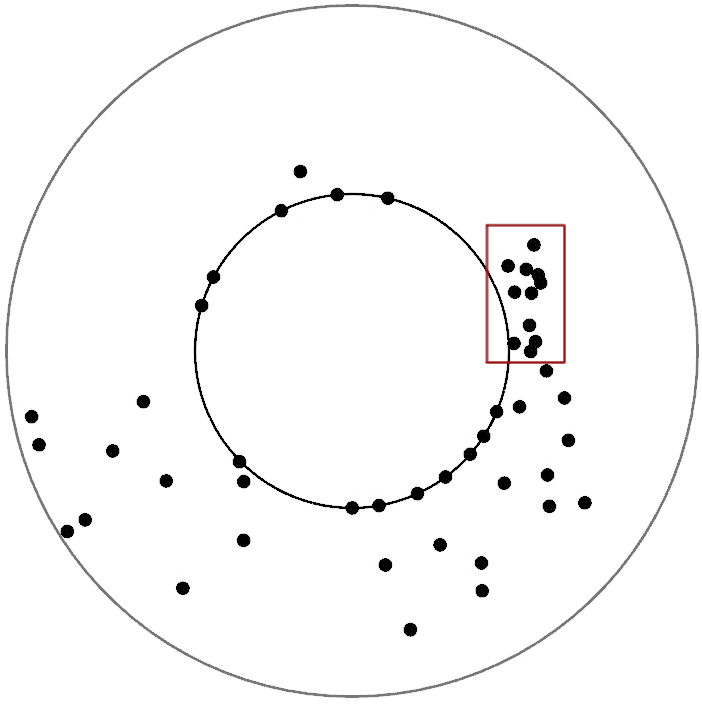}
    \caption{}\label{fig:densityB}
\end{subfigure}
\hfill
\begin{subfigure}{.22\linewidth}
    \centering
    \includegraphics[scale=0.5]{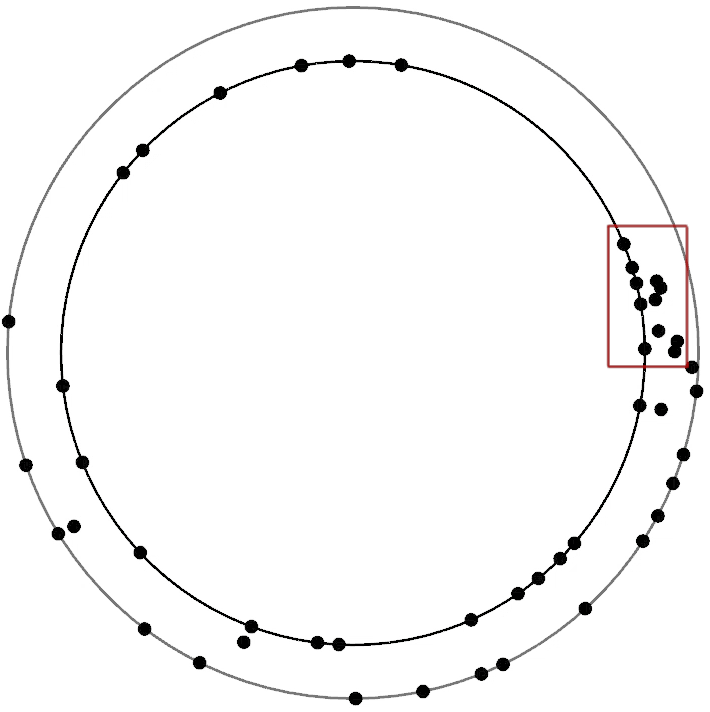}  
    \caption{}\label{fig:densityC}
\end{subfigure}
\hfill
\begin{subfigure}{.22\linewidth}
    \centering
    \includegraphics[scale=0.5]{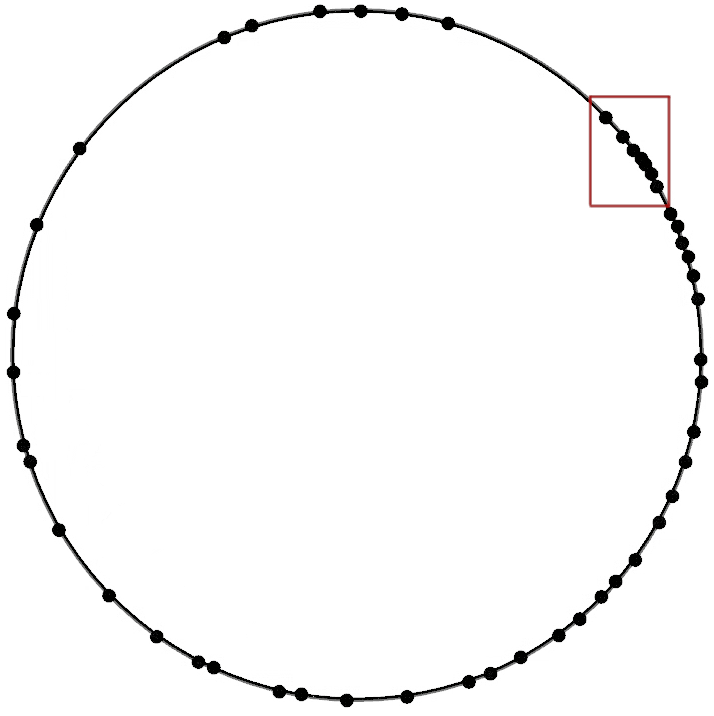}
    \caption{}\label{fig:densityD}
\end{subfigure}
\RawCaption{\caption{Local clumping within a non-uniform dataset.}
\label{fig:density1}}
\end{figure}

\bigskip

\begin{figure}[H]
\centering
\begin{subfigure}{.22\linewidth}
    \centering
    \includegraphics[scale=0.5]{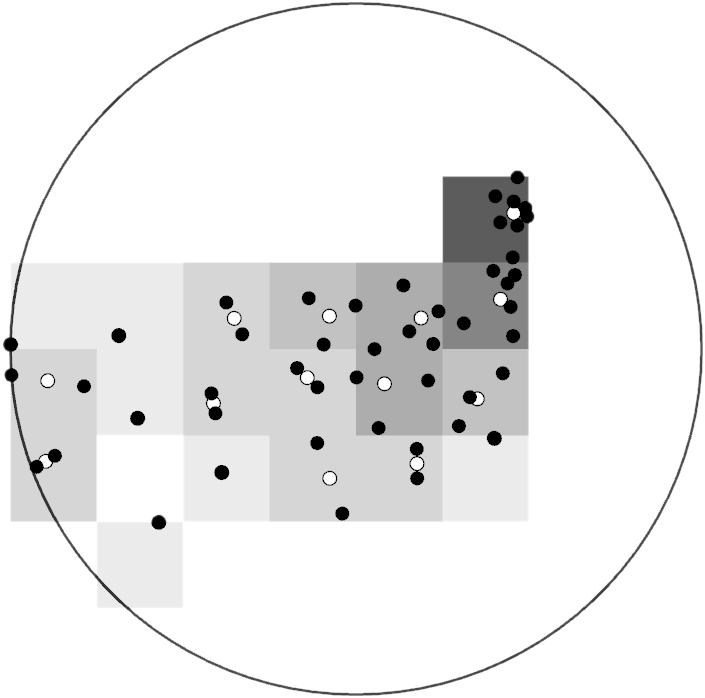}
    \caption{}\label{fig:densityE}
\end{subfigure}
\hfill
\begin{subfigure}{.22\linewidth}
    \centering
    \includegraphics[scale=0.5]{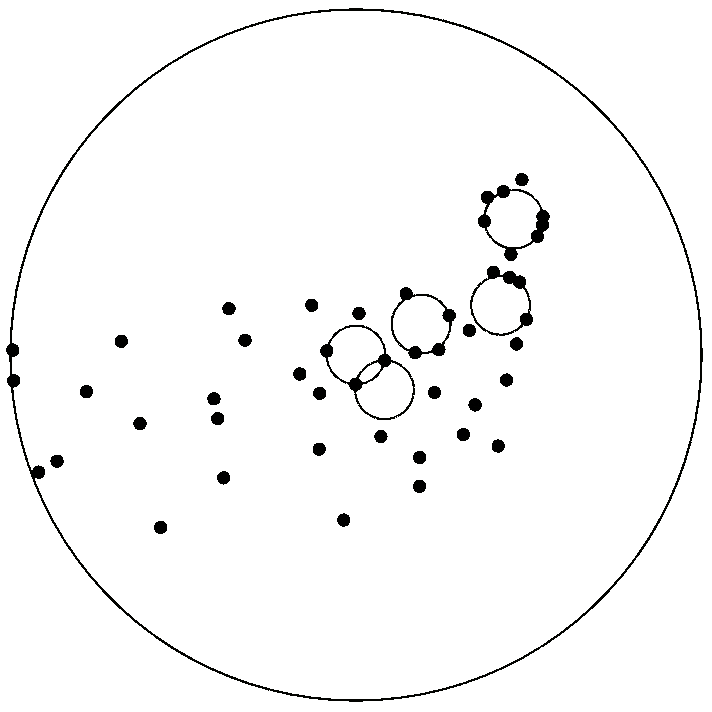}
    \caption{}\label{fig:densityF}
\end{subfigure}
\hfill
\begin{subfigure}{.22\linewidth}
    \centering
    \includegraphics[scale=0.5]{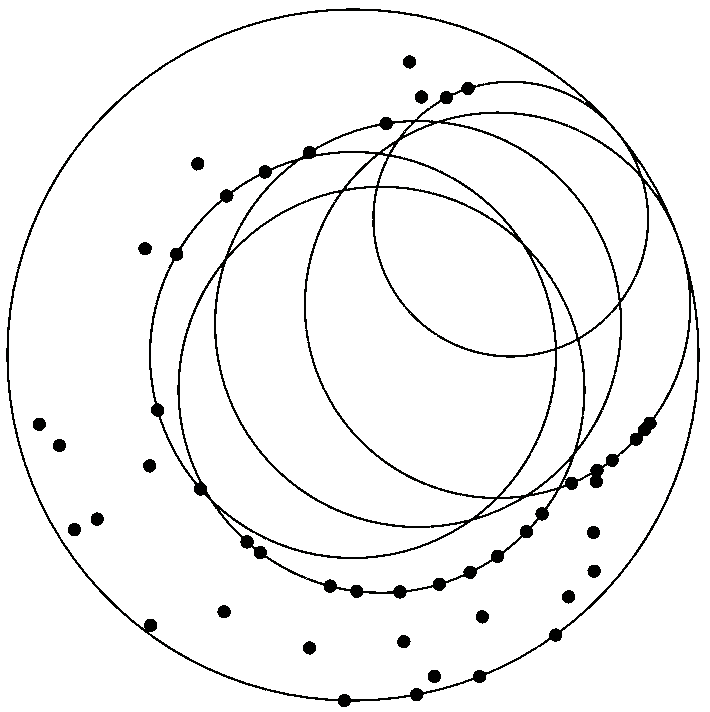}
    \caption{}\label{fig:densityG}
\end{subfigure}
\hfill
\begin{subfigure}{.22\linewidth}
    \centering
    \includegraphics[scale=0.5]{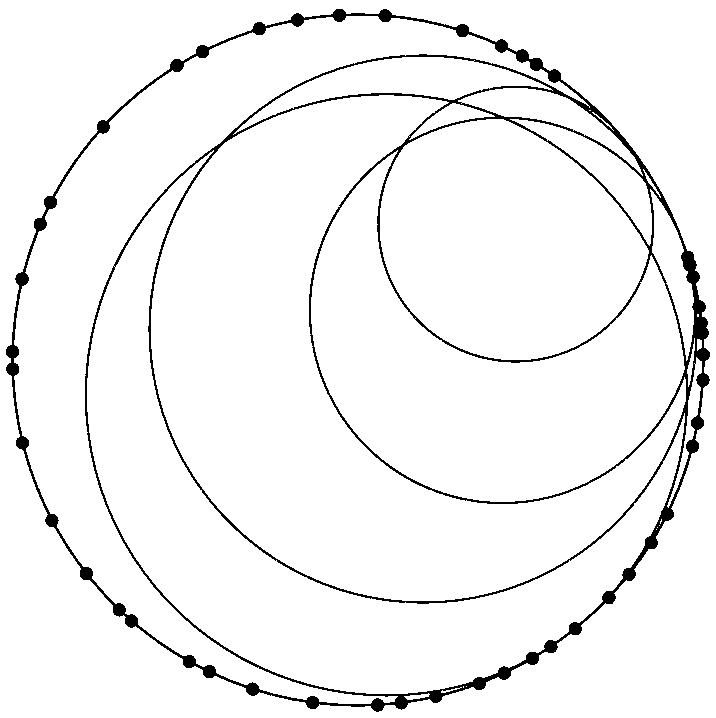}
    \caption{}\label{fig:densityH}
\end{subfigure}

\RawCaption{\caption{Using bubbles}
\label{fig:density2}}
\end{figure}

\bigskip

Some results of our N-body simulation with bubbles implemented are shown in Figure \ref{fig:density-grid-results} and Figure \ref{fig:density-att48-results}. These additional bubbles result in much better approximations compared to the simple N-body approach described earlier.

\begin{figure}[H]
    \centering
    \begin{minipage}{0.25\textwidth}
        \centering
        \includegraphics[scale=0.75]{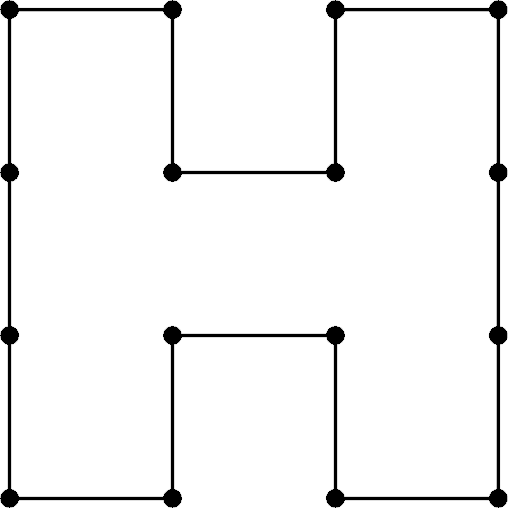}
    \end{minipage}
    \begin{minipage}{0.5\textwidth}
        \begin{footnotesize}
        \centering
        \renewcommand{\tabularxcolumn}[1]{ >{\arraybackslash}b{#1}}
        \begin{tabularx}{1.1\textwidth}{Z p{2cm}}
        Optimal Path Cost: & 16.000 \\
        Bubble N-body Path Cost: & 16.000 \\
        \textbf{Percent Error:} & \textbf{0.000\%} \\ \cmidrule(lr){1-2}
        Average Runtime: & 6.69 s
        \end{tabularx}
        \end{footnotesize}
    \end{minipage}
    \caption{Bubble N-body result for 4x4 grid instance.}
    \label{fig:density-grid-results}
\end{figure}

\begin{figure}[H]
    \centering
    \begin{minipage}{0.25\textwidth}
        \centering
        \includegraphics[scale=0.75]{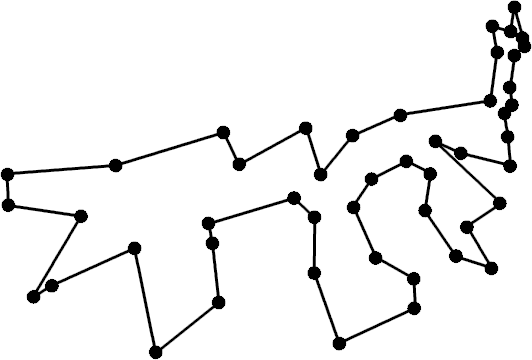}
    \end{minipage}
    \begin{minipage}{0.5\textwidth}
        \begin{footnotesize}
        \centering
        \renewcommand{\tabularxcolumn}[1]{ >{\arraybackslash}b{#1}}
        \begin{tabularx}{1.1\textwidth}{Z p{2cm}}
        Optimal Path Cost: & 33523.708 \\
        Bubble N-body Path Cost: & 35725.371 \\
        \textbf{Percent Error:} & \textbf{6.567\%} \\ \cmidrule(lr){1-2}
        Average Runtime: & 6.77 s
        \end{tabularx}
        \end{footnotesize}
    \end{minipage}
    \caption{Bubble N-body result for \textit{att48} instance.}
    \label{fig:density-att48-results}
\end{figure}

\subsection{Comparison Of The Different Method Variations}
Table \ref{tab:results-summary} summarizes the results of the different variations across different TSP instances. These results are shown here in an effort to highlight the improvements provided by the different variations of our N-body approach, as well as to showcase the potential future variations may have. From these results, we see that bubbles do not always provide an improvement, as seen with the \textit{ch150} instance. Additional results for the \textit{bayg29} and \textit{ch150} instances are provided in Appendix A. 

\begin{table}[H]
    \centering
    \caption{Summary of N-body results for different instances.}
    \label{tab:results-summary}
    \renewcommand{\tabularxcolumn }[1]{ >{\arraybackslash}b{#1}}
    \rowcolors{1}{tableShade}{white}
    \begin{tabularx}{0.95\textwidth}{@{}YYYYY@{}}
        \rowcolor{white}
        \textbf{Instance} & \textbf{Number of Cities} & \textbf{Simple N-body Percent Error} & \textbf{Pressure N-body Percent Error} & \textbf{Bubble N-body Percent Error}\\ \midrule
        4x4 Grid & 16 & 10.355\% & 10.355\% & \textbf{0.000\%}\\
        att48 & 48 & 15.927\% & 10.272\% & \textbf{6.576\%}\\
        bayg29 & 29 & 15.413\% & 8.227\% & \textbf{3.445\%}\\
        ch150 & 150 & 32.193\% & \textbf{15.636\%} & 29.118\%\\
        \bottomrule
        \hfill
    \end{tabularx}
    \hfill
\end{table}

\section{Conclusion}
Considering the early results presented in this text and the direction of recent improvements, we believe this is a promising approach. So far, we are able to rapidly obtain good solutions. Additionally, this approach showcases that perhaps solutions or better approximations can be found by approaching the problem in a unique and novel way. By viewing abstract problems with a physical perspective, one can draw on the vocabulary and intuition built in the physical world. This will hopefully allow us to describe the theoretical ideas and characteristics of the problem with physical terms. Furthermore, our work may be the foundation to approaching other problems related to the Traveling Salesman Problem, such as the popular open millennium problem of whether or not $P=NP$.  

\section{Future Work}
The motivation behind our future work is the desire to obtain accurate results on larger datasets in less time. To do this, there are many directions we are considering. For instance, we find it necessary to perform some type of parameter optimization. These parameters include the parameters for our Lennard-Jones type force functions, wall strength, the rate of change of wall radius, the simulation time step, the lower and upper bounds of the pressure range, local density cutoffs, and the number of bubbles to insert. Understanding the relationships between these parameters and the impact they have on the final result is important in trying to improve results. Another facet to consider is our implementation of the N-body simulation. Currently we are utilizing a brute-force approach in the N-body simulation. However, there are other N-body implementations that may allow us to reduce the runtime of our simulations. We also need to investigate other force functions. Perhaps there exists other force functions that are better suited for this approach than the Lennard-Jones type force functions we are currently using. In a more interesting direction, we would like to explore moving the simulation into higher-dimensional spaces. At the moment, our simulation only exists in a two-dimensional space. Perhaps letting the simulation run in higher-dimensional spaces will produce better results and allow us to tackle multi-weighted graphs. Current thoughts revolve around the idea of using a torus to facilitate the process of taking an $n$-dimensional system and constricting it to a one-dimensional path. Finally, we would like to extrapolate and investigate the relationships between the physical parameters and the abstract ideas of the problem. 

\printbibliography
\section*{Acknowledgements}
Tarleton State University's College of Science and Technology for research and travel funds. Tarleton State University's Mathematics Department for valuable suggestions and feedback. Tarleton State University's High Performance Computing Lab for space and computational resources. Nvidia for providing us with graphical processing units (GPUs) used for computational acceleration.

\section*{Supplemental Material}
For further reading, data, results, and videos visit \url{https://seayjohnny.github.io/NBodyTSP}

\newpage
\appendix
\section{More Results}
\begin{figure}[H]
    \centering
    \begin{minipage}{0.25\textwidth}
        \centering
        \includegraphics[scale=0.65]{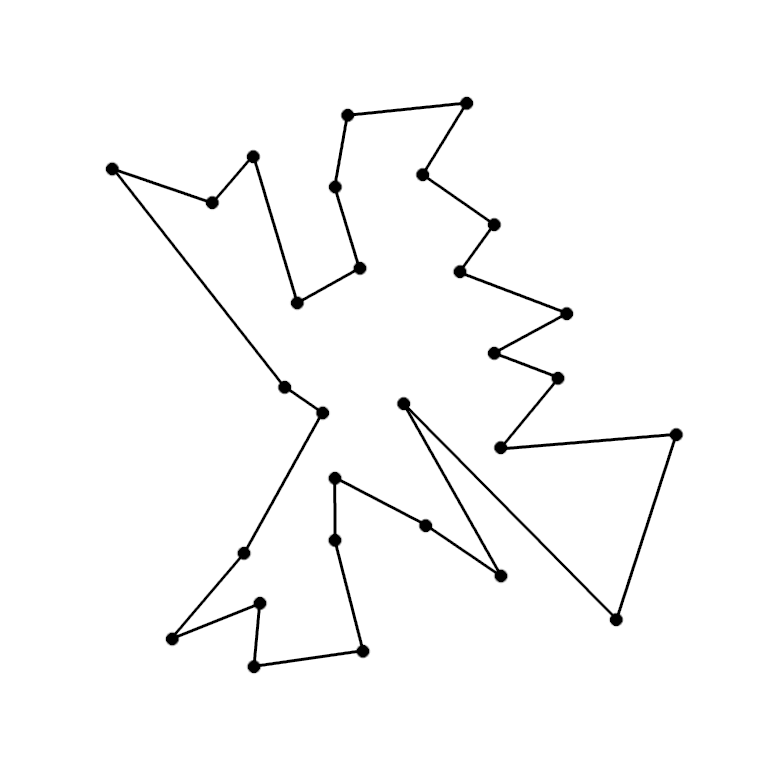}
        \caption{Simple N-body result for \textit{bay29} instance.}
    \end{minipage}
    \hspace{5mm}
    \begin{minipage}{0.5\textwidth}
        \begin{footnotesize}
        \centering
        \renewcommand{\tabularxcolumn}[1]{ >{\arraybackslash}b{#1}}
        \begin{tabularx}{\textwidth}{Z p{2cm}}
        Optimal Path Cost: & 9291.353 \\
        Simple N-body Path Cost: & 10723.422 \\
        \textbf{Percent Error:} & \textbf{15.413\%}
        \end{tabularx}
        \end{footnotesize}
    \end{minipage}
\end{figure}

\begin{figure}[H]
    \centering
    \begin{minipage}{0.25\textwidth}
        \centering
        \includegraphics[scale=0.65]{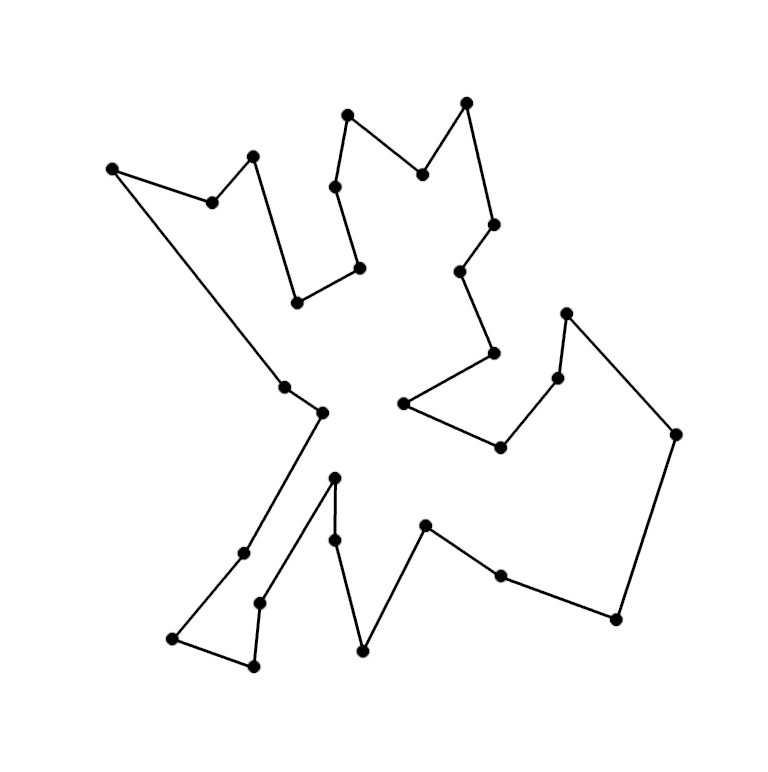}
        \caption{Pressure N-body result for \textit{bay29} instance.}
    \end{minipage}
    \hspace{5mm}
    \begin{minipage}{0.5\textwidth}
        \begin{footnotesize}
        \centering
        \renewcommand{\tabularxcolumn}[1]{ >{\arraybackslash}b{#1}}
        \begin{tabularx}{\textwidth}{Z p{2cm}}
        Optimal Path Cost: & 9291.353 \\
        Pressure N-body Path Cost: & 10055.722 \\
        \textbf{Percent Error:} & \textbf{8.227\%}
        \end{tabularx}
        \end{footnotesize}
    \end{minipage}
\end{figure}

\begin{figure}[H]
    \centering
    \begin{minipage}{0.25\textwidth}
        \centering
        \includegraphics[scale=0.65]{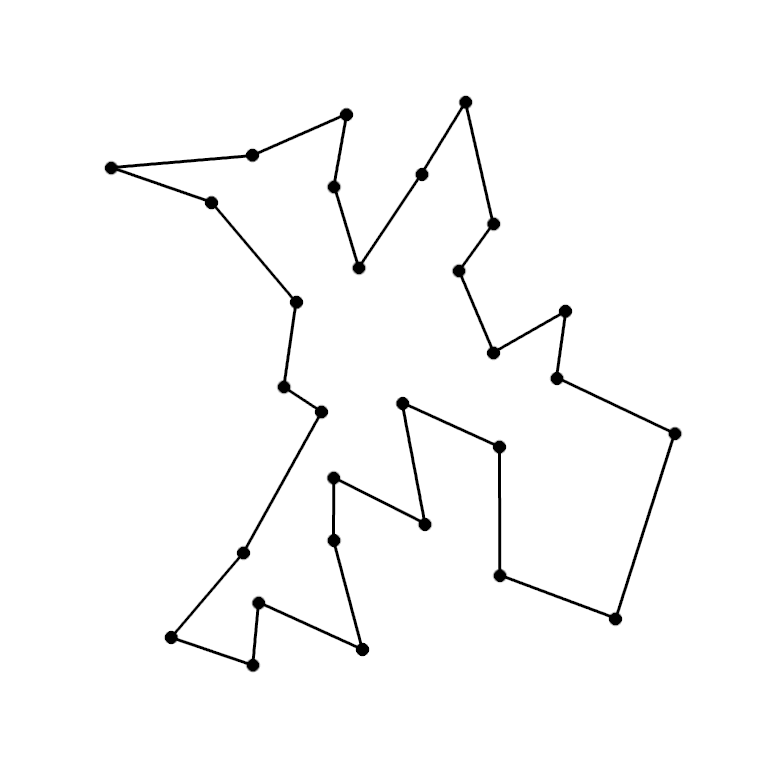}
        \caption{Bubble N-body result for \textit{bay29} instance.}
    \end{minipage}
    \hspace{5mm}
    \begin{minipage}{0.5\textwidth}
        \begin{footnotesize}
        \centering
        \renewcommand{\tabularxcolumn}[1]{ >{\arraybackslash}b{#1}}
        \begin{tabularx}{\textwidth}{Z p{2cm}}
        Optimal Path Cost: & 9291.353 \\
        Bubble N-body Path Cost: & 35725.371 \\
        \textbf{Percent Error:} & \textbf{3.445\%}
        \end{tabularx}
        \end{footnotesize}
    \end{minipage}
\end{figure}

\newpage
\begin{figure}[H]
    \centering
    \begin{minipage}{0.25\textwidth}
        \centering
        \includegraphics[scale=0.65]{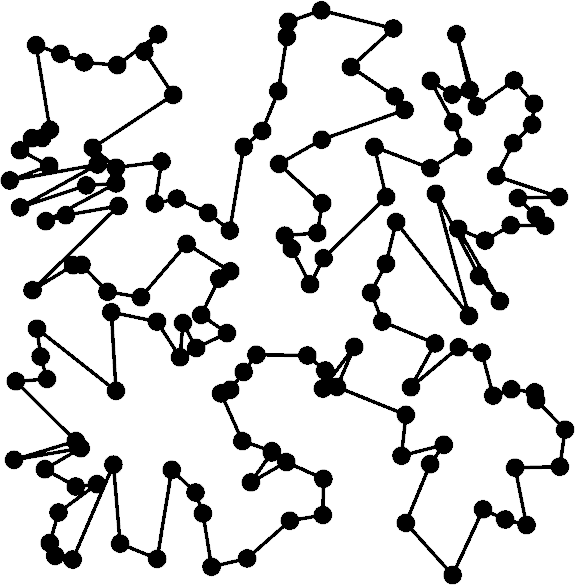}
        \caption{Simple N-body result for \textit{ch150} instance.}
    \end{minipage}
    \hspace{5mm}
    \begin{minipage}{0.5\textwidth}
        \begin{footnotesize}
        \centering
        \renewcommand{\tabularxcolumn}[1]{ >{\arraybackslash}b{#1}}
        \begin{tabularx}{\textwidth}{Z p{2cm}}
        Optimal Path Cost: & 6532.28 \\
        Simple N-body Path Cost: & 8261.32 \\
        \textbf{Percent Error:} & \textbf{26.47\%}
        \end{tabularx}
        \end{footnotesize}
    \end{minipage}
    \label{fig:ch150-simple}
\end{figure}

\begin{figure}[H]
    \centering
    \begin{minipage}{0.25\textwidth}
        \centering
        \includegraphics[scale=0.65]{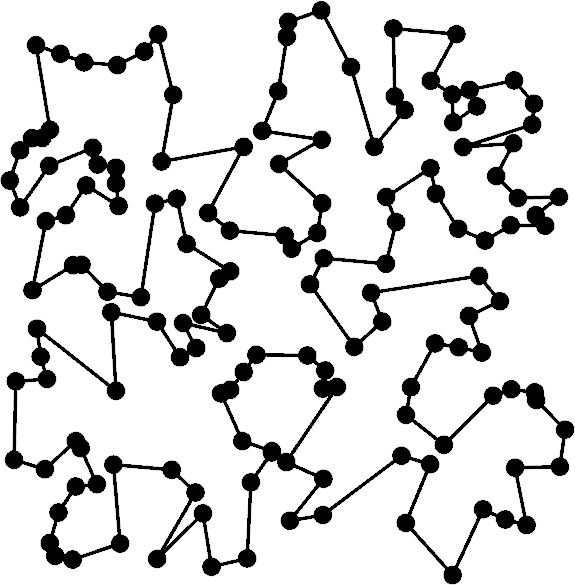}
        \caption{Pressure N-body result for \textit{ch150} instance.}
    \end{minipage}
    \hspace{5mm}
    \begin{minipage}{0.5\textwidth}
        \begin{footnotesize}
        \centering
        \renewcommand{\tabularxcolumn}[1]{ >{\arraybackslash}b{#1}}
        \begin{tabularx}{\textwidth}{Z p{2cm}}
        Optimal Path Cost: & 6532.28 \\
        Pressure N-body Path Cost: & 7553.66 \\
        \textbf{Percent Error:} & \textbf{15.64\%}
        \end{tabularx}
        \end{footnotesize}
    \end{minipage}
    \label{fig:ch150-pressure}
\end{figure}

\begin{figure}[H]
    \centering
    \begin{minipage}{0.25\textwidth}
        \centering
        \includegraphics[scale=0.65]{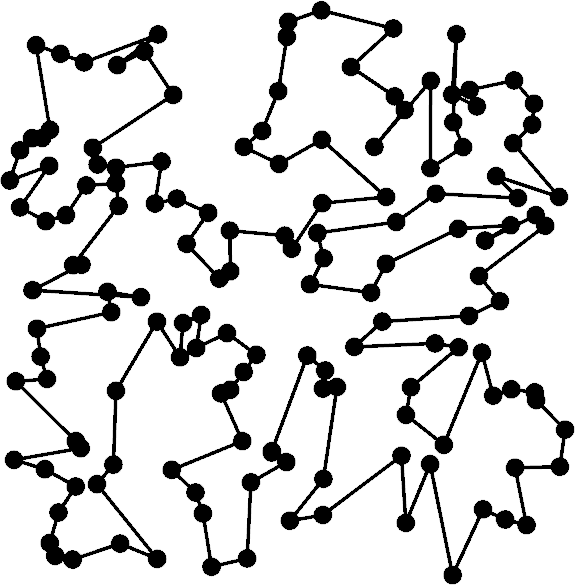}
        \caption{Bubble N-body result for \textit{ch150} instance.}
    \end{minipage}
    \hspace{5mm}
    \begin{minipage}{0.5\textwidth}
        \begin{footnotesize}
        \centering
        \renewcommand{\tabularxcolumn}[1]{ >{\arraybackslash}b{#1}}
        \begin{tabularx}{\textwidth}{Z p{2cm}}
        Optimal Path Cost: & 6532.28 \\
        Bubble N-body Path Cost: & 8270.07 \\
        \textbf{Percent Error:} & \textbf{26.60\%}
        \end{tabularx}
        \end{footnotesize}
    \end{minipage}
    \label{fig:ch150-density}
\end{figure}

\newpage
\section{Reparameterizing Lennard-Jones Force Functions}
Given a Lennard-Jones potential function, the negative of its derivative is the corresponding Lennard-Jones force function (LJF), which can be written in the following form.
\begin{equation}
\label{LJF}
F(r) = \frac{G}{r^q} - \frac{H}{r^p},\ r>0,
\end{equation}

\smallskip

\noindent  where $G$, $H$, $q$, and $p$ are positive real numbers, and $q > p$.   

\bigskip

Equation~\eqref{LJF} expresses the force $F(r)$ acting on two particles based on their distance $r$ from each other, with positive and negative values corresponding to repulsive and attractive forces, respectively.  The parameters $G$, $H$, $q$, and $p$ govern the relative strength of the repulsive and attractive terms in $F$, but unfortunately these parameters are difficult to interpret geometrically.  The goal of this section is to reparameterize Lennard-Jones force functions using more geometrically meaningful parameters.

\bigskip

\begin{figure}[htbp]
    \centering
    \includegraphics{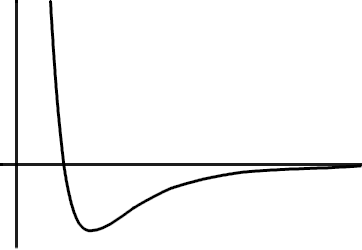}
    \caption{General shape of Lennard-Jones force function.}
    \label{fig:LJFfigure}
\end{figure}{}

Proposition~\ref{curveshape} below shows that every LJF essentially has the same shape, as displayed in Figure \ref{fig:LJFfigure}. Arbitrarily large repulsive forces occur at small distances $r$, followed by a unique root $F(L)=0$ and a unique absolute minimum $F(r_\text{min})=-M$.  As $r$ increases beyond $r_\text{min}$, $F(r)$ asymptotically converges to zero.  In other words, there is an equilibrium distance $L$ at which the force is zero, followed by a point $r_\text{min}$ where the maximal attractive force is attained, after which the attractive force decays to zero.  As a side note, Proposition~\ref{curveshape} also shows that every LJF is convex until reaching a unique inflection point $r_\text{infl}$ and is concave afterwards.

%The parameters $G$, $H$, $q$, and $p$ therefore determine the more geometrically meaningful parameters $L$, $r_\text{min}$, $M$, and a fourth parameter $\delta = q-p$, which will be shown to govern the rate at which the attractive force decays.  The following theorem states that there is a one-to-one correspondence between the parameter vector $(G,H,q,p)$ and $(L,r_\text{min},M,\delta)$, providing a reparameterization of the Lennard-Jones force functions.
%
%

%Regardless of the values of these parameters, $F$ satisfies the following.

\begin{proposition}
\label{curveshape}
Let $F$ be a Lennard-Jones force function determined by parameters \\ $G,H,q,p > 0$, where  $q > p$.
\begin{enumerate}
\item \label{limit1}
$\lim_{r \to 0^+} F(r) = \infty$
\item 
$\lim_{r \to \infty} F(r) = 0$
\item \label{limit2}
$F(L)=0$ is uniquely satisfied by 
\[
L = \left(\frac{G}{H}\right)^\frac{1}{q-p}
\]
\item  \label{rmin}
$F$ attains a unique absolute minimum at 
\[
r_\text{min}=\left(\frac{q}{p}\right)^\frac{1}{q-p}L
\]
\item \label{M}
The absolute minimum value is $F(r_\text{min}) = -M$, where
\[
M ={H}{\left[\left(\frac{q}{p}\right)^{\frac{1}{q-p}} L\right]^{-p}}\left[1-\frac{p}{q}\right]
\]

\item
$F$ has a unique inflection point at 
\[
r_\text{infl}=\left(\frac{q+1}{p+1}\right)^\frac{1}{q-p} r_\text{min}
\]
\item
$L < r_\text{min} < r_\text{infl}$.

\end{enumerate}
\end{proposition}

\smallskip

\begin{proof}
Properties~\eqref{limit1} through \eqref{limit2} follow from the factorization
\begin{equation}\label{factored}
F(r) = \frac{1}{r^p}\left(\frac{G}{r^{q-p}} - H\right).
\end{equation}

\noindent  The first and second derivatives of $F$ are
\begin{align}
F'(r) & = -q\frac{G}{r^{q+1}} + p \frac{H}{r^{p+1}} \\
& = \frac{1}{r^{p+1}}\left(-q\frac{G}{r^{q-p}} + pH \right)  \label{firstderiv} \\
F''(r) & = q(q+1)\frac{G}{r^{q+2}} - p(p+1) \frac{H}{r^{p+2}} \\
& = \frac{1}{r^{p+2}}\left(q(q+1)\frac{G}{r^{q-p}} - p(p+1) H \right) \label{secondderiv}.
\end{align}

\noindent  Equation~\eqref{firstderiv} shows that $F$ is strictly decreasing on the interval $(0,r_\text{min})$ and strictly increasing on $(r_\text{min},\infty)$, establishing the existence of a unique minimum at 

\[
r_\text{min}=\left(\frac{qG}{pH}\right)^\frac{1}{q-p}=\left(\frac{q}{p}\right)^\frac{1}{q-p}L > L.
\]

\smallskip

\noindent  Similarly, equation~\eqref{secondderiv} shows that $F$ is convex on $(0,r_\text{infl})$ and concave on $(r_\text{infl},\infty)$, with a unique inflection point at 

\[
r_\text{infl}=\left(\frac{q(q+1)G}{p(p+1)H}\right)^\frac{1}{q-p} = \left(\frac{q+1}{p+1}\right)^\frac{1}{q-p} r_\text{min} > r_\text{min}.
\]

\smallskip

All that remains is showing $F(r_\text{min}) = -M$.  This is easily verified by using $G=HL^{q-p}$ to rewrite Equation~\eqref{factored} as follows:

\begin{equation}
\label{newform}
F(r) = \frac{H}{r^p}\left[\left(\frac{L}{r}\right)^{q-p} - 1\right].
\end{equation}
\end{proof}

%\footnote{In this equation, we have used subscripts to make the dependency on $H$ and $L$ explicit, which is used later.}

We now have three geometrically meaningful parameters $L$, $r_\text{min}$, and $M$.  A fourth parameter of interest, $\delta = q-p$, will later be shown to control the decay rate of the attractive force.  Proposition~\ref{curveshape} shows that the parameter vector $(G,H,q,p)$ determines the values of the parameter vector $(L, r_\text{min}, M, \delta)$.  Conversely, the following proposition shows that $(L, r_\text{min}, M, \delta)$ determines $(G,H,q,p)$.

\begin{proposition}
\label{converse}
Assume $L,r_\text{min},M,\delta > 0$, where $L < r_\text{min}$.  Then there exist unique parameters $G,H,q,p > 0$, where $q>p$, such that the corresponding Lennard-Jones force function $F$ satisfies the following.

\begin{enumerate}
\item
$F(L)=0$
\item
$F$ attains an absolute minimum at $F(r_\text{min}) = -M$
\item
$\delta = q-p$.
\end{enumerate}
\end{proposition}

\begin{proof}
The following equations sequentially define $p$, $q$, $H$, and $G$ in terms of $L$, $r_\text{min}$, $M$, and $\delta$.
\begin{align}
p & = \frac{\delta}{\left(\frac{r_\text{min}}{L}\right)^\delta - 1} \label{peq} \\
q & = p + \delta   \label{qeq} \\
H & = M\left[\left(\frac{q}{p}\right)^{\frac{1}{q-p}} L\right]^{p}\left[1-\frac{p}{q}\right]^{-1} \label{Heq} \\
G & = HL^{q-p} \label{Geq}
\end{align}

\noindent  These equations are derived by algebraically manipulating the equations from parts \eqref{limit2} through \eqref{M} of Proposition~\ref{curveshape}.\footnote{This observation establishes uniqueness of $(G,H,q,p)$.}  In practice, it is necessary to apply these equations in the order listed.  For instance, $p$ must be computed using \eqref{peq} before it can be used in \eqref{qeq} to compute $q$.  Similarly, $p$ and $q$ must both be computed before they can be used in $\eqref{Heq}$ to compute $H$.

\bigskip

Because $L < r_\text{min}$, Equation~\eqref{peq} defines a positive value for  $p$.   It is then easy to verify that $q, H, G > 0$ and $q > p$.  Therefore, Equations~\eqref{peq} through \eqref{Geq} define a parameter vector $(G,H,q,p)$ corresponding to a Lennard-Jones force function $F$.  Applying part \eqref{limit2} of Proposition~\ref{curveshape} to $F$ implies that it has a unique zero at

\[
\left(\frac{G}{H}\right)^\frac{1}{q-p} = \left(\frac{HL^{q-p}}{H}\right)^\frac{1}{q-p} = L.
\]

\smallskip

\noindent  Similarly, parts \eqref{rmin} and \eqref{M} of Proposition~\ref{curveshape} prove that $F(r_\text{min}) = -M$ is the absolute minimum.
\end{proof}

Propositions~\ref{curveshape} and \ref{converse} establish a one-to-one correspondence between the parameter sets 

\[
\{(G,H,q,p) \mid G,H,q,p > 0\text{ and }q > p\}\text{, and }
\]

\[
\{(L, r_\text{min}, M, \delta) \mid L,r_\text{min},M,\delta > 0\text{ and }L < r_\text{min}\},
\]

\noindent  reparameterizing Lennard-Jones force functions in terms of more geometrically meaningful parameters.  All that remains is to investigate the role of the shape parameter $\delta$, which requires the following lemma.

\begin{lemma}
\label{partiallemma}
If $F$ is a Lennard-Jones force function with parameters $(L,r_\text{min},M,\delta)$, and $r > r_\text{min}$, then

\[
\frac{\partial}{\partial \delta} F(r) < 0.
\]
\end{lemma}

\smallskip

\begin{proof}
First, we rewrite Equation~\eqref{newform} using the $(L,r_\text{min},M,\delta)$-parameterization

\begin{align}
F(r) & = \frac{M\left[\left(\frac{q}{p}\right)^{\frac{1}{q-p}} L\right]^{p}\left[1-\frac{p}{q}\right]^{-1}}{r^p}\left[\left(\frac{L}{r}\right)^{q-p} - 1\right] \\
\label{newparamform}  F(r) & = {M\left(\frac{r_\text{min}}{r}\right)^\frac{\delta}{\left(\tfrac{r_\text{min}}{L}\right)^\delta - 1}}\frac{\left[\left(\frac{L}{r}\right)^{\delta} - 1\right]}{\left[1-\left(\frac{L}{r_\text{min}}\right)^\delta\right]}.
\end{align}

\smallskip

\noindent  The parameters $L$ and $M$ are simply horizontal and vertical scale parameters, so without loss of generality, assume $L = M = 1$.  Because $\log$ is an increasing function, $\frac{\partial}{\partial \delta} F(r) < 0$ if and only if $\frac{\partial}{\partial \delta} \log[-F(r)] > 0$.

\[
-F(r) = \left(\frac{r_\text{min}}{r}\right)^{\frac{\delta}{r_\text{min}^\delta - 1}}\frac{\left[1-\left(\frac{1}{r}\right)^{\delta}\right]}{\left[1-\left(\frac{1}{r_\text{min}}\right)^\delta\right]}
\]
\begin{align*}
\log[-F(r)] &  = \frac{\delta}{r_\text{min}^\delta - 1}\log\left(\frac{r_\text{min}}{r}\right)+\log\left[1-\left(\frac{1}{r}\right)^{\delta}\right]-\log \left[ 1-\left(\frac{1}{r_\text{min}}\right)^\delta\right] \\
\frac{\partial}{\partial \delta} \log[-F(r)] & = \frac{r_\text{min}^\delta - 1 -\delta r_\text{min}^\delta \log r_\text{min}}{(r_\text{min}^\delta - 1)^2}\log\left(\frac{r_\text{min}}{r}\right)+\frac{\log r}{r^\delta-1}-\frac{\log r_\text{min}}{r_\text{min}^\delta-1}. \\
\end{align*}

\noindent  After considerable algebraic simplification, we have

\[
\frac{\partial}{\partial \delta} \log[-F(r)] = \frac{\delta r_\text{min}^\delta (r^\delta -1)(\log r_\text{min})(\log r - \log r_\text{min})-(r^\delta - r_\text{min}^\delta)
(r_\text{min}^\delta -1)\log r}{(r_\text{min}^\delta-1)^2 (r^\delta -1)}.
\]

\noindent  Keeping in mind that $r > r_\text{min} > L =1$, and $\delta > 0$, this denominator is positive, so it suffices to show the numerator is positive.  Making the substitution $x=r_\text{min} > 1$, and $y=r/r_\text{min} > 1$, the numerator is
\begin{align*}
 & \delta x^\delta(x^\delta y^\delta -1)(\log x)(\log y) - (x^\delta y^\delta - x^\delta)(x^\delta -1) \log xy \\
= & \delta^{-1} x^\delta [(x^\delta y^\delta -1)(\log x^\delta)(\log y^\delta) - (x^\delta -1)(y^\delta -1)\log x^\delta y^\delta].
\end{align*}

\noindent  Making a second substitution $a = \log x^\delta > 0$ and $b=\log y^\delta >0$, and noting that $\delta^{-1} x^\delta >0$, it suffices to show that the following expression is positive:

\[
(e^{a+b} -1)ab -(e^a -1)(e^b-1)(a+b).
\]

\noindent  Rewriting this expression using the exponential function's Maclaurin series, whose radius of convergence is infinite, and factoring out $(a+b)ab$ yields

\[
(a+b)ab \left(\sum_{n=1}^\infty \frac{(a+b)^{n-1}}{n!} - \sum_{i=1}^\infty \frac{a^{i-1}}{i!} \sum_{j=1}^\infty \frac{b^{j-1}}{j!}  \right)
\]

\[
=(a+b)ab\left(\sum_{n=1}^\infty \sum_{m=0}^{n-1}  \frac{(n-1)!}{n!m!(n-m-1)!}a^{m}b^{n-m-1} - \sum_{i=1}^\infty \sum_{j=1}^\infty  \frac{a^{i-1}}{i!}  \frac{b^{j-1}}{j!}  \right),
\]

\noindent  by the Binomial Theorem.  The second series is absolutely convergent, so we can rewrite the order of summation to obtain

\[
(a+b)ab\left(\sum_{n=1}^\infty \sum_{m=0}^{n-1}  \frac{1}{nm!(n-m-1)!}a^{m}b^{n-m-1} - \sum_{n=1}^\infty \sum_{m=0}^{n-1}  \frac{a^m}{(m+1)!}  \frac{b^{n-m-1}}{(n-m)!}  \right)
\]
\[
=(a+b)ab\sum_{n=1}^\infty \sum_{m=0}^{n-1} \frac{m(n-m-1)}{n(m+1)!(n-m)!}a^{m}b^{n-m-1} > 0.
\]

\end{proof}

We are now prepared to discuss the impact of $\delta$ on the decay of the attractive force.  After an LJF attains its minimum $F(r_\text{min}) = -M$, it monotonically converges to $0$.  Given $\varepsilon \in (0,1)$, the Intermediate Value Theorem guarantees the existence of  some unique $r_\varepsilon > r_\text{min}$, such that $F(r_\varepsilon) = -\varepsilon M$.  For example, if $\varepsilon = 0.01$, then $r_\varepsilon$ is the point where the attractive force has decayed to 1\% of its maximum magnitude.  Choosing the value of $r_\varepsilon$ provides control over the decay rate of the attractive force.  For example, choosing $r_{0.9}$ to be much greater than $r_\text{min}$ means that the attractive force is still at 90\% of its maximum magnitude for values of $r$ much larger than $r_\text{min}$, i.e., the attractive force is decaying very slowly.  On the other hand, choosing $r_{0.01}$ to be very close to $r_\text{min}$ implies a very rapid decay of the attractive force.  The following theorem illustrates the importance of the parameter $\delta$ in determining $r_\varepsilon$.

\begin{theorem}
Consider a Lennard-Jones force function $F$ with fixed parameters $L$, $r_\text{min}$, and $M$, assume $0 < \varepsilon < 1$, and define the function

\[
T(r) = \left(\frac{r_\text{min}}{r}\right)^\frac{1}{\log\left(\frac{r_\text{min}}{L}\right)}\frac{\log\left(\frac{r}{L}\right)}{\log\left(\frac{r_\text{min}}{L}\right)}.
\]

\begin{enumerate}
\item  \label{Rexist}
There exists $R_\varepsilon > r_\text{min}$, such that $T(R_\varepsilon) = \varepsilon$.
\item \label{possiblereps}
The set of possible values of $r_\varepsilon$ is $(R_\varepsilon,\infty)$.
\item \label{delta1to1A}
There is a one-to-one correspondence between the values of $\delta$ and $r_\varepsilon$.
\item
$\lim_{\delta \to 0^+} r_\varepsilon = R_\varepsilon$
\item \label{delta1to1B}
$\lim_{\delta \to \infty} r_\varepsilon = \infty$.
\end{enumerate}
\end{theorem}

\begin{proof}  Because $L < r_\text{min}$, it is routine to verify that $T$ is decreasing on $[r_\text{min},\infty)$, with $T(r_\text{min}) = 1$ and $\lim_{r \to \infty} T(r) = 0$.  Property~\eqref{Rexist} then follows from the Intermediate Value Theorem.  Given $r > r_\text{min}$, applying L'Hopital's rule to Equation~\eqref{newparamform} provides values for the following limits:

\begin{equation}\label{deltalimit1}
\lim_{\delta \to \infty} F(r) = -M
\end{equation}
\begin{equation}\label{deltalimit2}
\lim_{\delta \to 0^+} F(r) = -M\left(\frac{r_\text{min}}{r}\right)^\frac{1}{\log\left(\frac{r_\text{min}}{L}\right)}\frac{\log\left(\frac{r}{L}\right)}{\log\left(\frac{r_\text{min}}{L}\right)}.
\end{equation}

\smallskip

\noindent  Furthermore, by Lemma~\ref{partiallemma},   $F(r)$ is a monotonically decreasing function of $\delta$.  Therefore, there exists $\delta > 0$, such that $F(r) = - \varepsilon M$ if and only if 

\[
-M < -\varepsilon M < -M \left(\frac{r_\text{min}}{r}\right)^\frac{1}{\log\left(\frac{r_\text{min}}{L}\right)}\frac{\log\left(\frac{r}{L}\right)}{\log\left(\frac{r_\text{min}}{L}\right)},
\]

\smallskip

\noindent  which holds if and only if $T(r) < \varepsilon$.  This inequality is equivalent to $r > R_\varepsilon$, since $T$ is decreasing on $[r_\text{min},\infty)$, which establishes Property~\eqref{possiblereps}.  Properties~\eqref{delta1to1A} through \eqref{delta1to1B} now follow from Lemma~\ref{partiallemma} and the limits \eqref{deltalimit1} and \eqref{deltalimit2}.
\end{proof}

In summary, this section has shown how to reparameterize Lennard-Jones force functions in terms of more geometrically meaningful parameters.  Given desired values $r_\text{min} > L > 0$ and $M > 0$, a Lennard-Jones force function can always be specified with equilibrium length $L$ and maximal attractive force $F(r_\text{min}) = -M$.  Choosing $r_\text{min}$ near $L$ results in an attractive force that ramps up quickly, while choosing $r_\text{min}$ much larger than $L$ corresponds to a slow increase of the attractive force's magnitude.  The decay of the attractive force is then specified as follows:

\begin{enumerate}
\item
Choose some $\varepsilon \in (0,1)$.
\item
Compute $R_\varepsilon$ by solving $T(R_\varepsilon) = \varepsilon$ using a numerical method, such as Newton's method.
\item
Choose a desired value $r_\varepsilon \in (R_\varepsilon,\infty)$.
\item
Use a numerical method and Equation~\eqref{newparamform} to solve $F(r_\varepsilon) = -\varepsilon M$ for $\delta$.
\item
Use Equations~\eqref{peq} through \eqref{Geq} to compute $p$, $q$, $H$, and $G$, which fully specify the Lennard-Jones force function.
\end{enumerate}

\noindent  Only one value of $r_\varepsilon$ can be specified, e.g., it is not possible to simultaneously choose arbitrary values for $r_{0.2}$ and $r_{0.01}$, but overall, this approach still provides a great deal of flexibility.  Hopefully this reparameterization allows researchers engaged in particle modeling to intuitively choose realistic parameter values for a variety of applications.

\end{document}